\documentclass[a4paper]{article}
\usepackage[a4paper,top=3cm,bottom=2cm,left=3cm,right=3cm,marginparwidth=1.75cm]{geometry}

\usepackage{amsmath}
\usepackage[shortlabels]{enumitem}
\usepackage{amssymb}
\usepackage{multirow}
\usepackage{mathtools}
\usepackage{amsthm}
\usepackage{graphicx}
\usepackage[colorinlistoftodos]{todonotes}
\usepackage[colorlinks=true, allcolors=blue]{hyperref}
\DeclareMathOperator{\PG}{PG}

\newcommand{\qbin}[2]{\genfrac{[}{]}{0pt}{}{#1}{#2}}
\usepackage{pgf,tikz,pgfplots, booktabs,subfig}
\usepackage{mathrsfs}
\usetikzlibrary{arrows}

\usepackage[english]{babel}
\usepackage[utf8x]{inputenc}
\usepackage[T1]{fontenc}

\newtheorem{definition}{Definition}[section]
\newtheorem{prop}[definition]{Proposition}
\newtheorem{theorem}[definition]{Theorem}
\newtheorem{mtheorem}[definition]{Main Theorem}

\newtheorem{lemma}[definition]{Lemma}
\newtheorem{gevolg}[definition]{Corollary}

\newtheorem{remark}[definition]{Remark}
\newtheorem{example}[definition]{Example}
\newtheorem{notation}[definition]{Notation}

\renewenvironment{proof}[1][\noindent Proof]{{\par\pushQED{\qed}\itshape #1\@. }}{\popQED}

\title{Maximal sets of $k$-spaces pairwise intersecting in at least a $(k-2)$-space}
\author{Jozefien D'haeseleer\thanks{Department of Mathematics: Analysis, Logic and Discrete Mathematics, Ghent University, Krijgslaan 281, Building S8, 9000 Gent, Flanders, Belgium; jozefien.dhaeseleer@ugent.be, leo.storme@ugent.be.}\and Giovanni Longobardi\thanks{Department of Management and Engineering, University of Padua,
Stradella S. Nicola 3,
36100 Vicenza,
Italy; giovanni.longobardi@unipd.it.} \and Ago-Erik Riet\thanks{ Institute of Mathematics and Statistics, University of Tartu, Narva mnt 18, 51009 Tartu, Estonia; agoerik@ut.ee.} \and Leo Storme\footnotemark[1]}\date{}
\begin{document}
\maketitle
\begin{abstract}
 
In this paper, we analyze the structure of maximal sets of  $k$-dimensional spaces in $\PG(n,q)$ pairwise intersecting in at least a $(k-2)$-dimensional space, for $3 \leq k\leq n-2$. We give an overview of the largest examples of these sets with size more than $f(k,q)=\max\{3q^4+6q^3+5q^2+q+1,\theta_{k+1}+q^4+2q^3+3q^2\}$.

\end{abstract}

\section{Introduction and preliminaries}

One of the classical problems in extremal set theory is to determine the size of the largest sets of pairwise non-trivially  intersecting subsets. In 1961, it was solved by Erd{\H{o }}s, Ko and Rado \cite{erdos_ko_rado}, and their result was improved by Wilson in 1984.

\begin{theorem}{\textbf{\cite{wilsonEKR}}}\label{thm1}
	Let $n, k$ and $t$ be positive integers and  suppose that $k \geq t \geq 1$ and $n \geq (t+1)(k-t+1)$. If $\mathcal{S}$ is a family of subsets of size $k$ in a set $\Omega$ with $|\Omega | = n$ , such that the elements of $\mathcal{S}$ pairwise intersect in at least $t$ elements, then $|\mathcal{S}| \leq \binom{n-t}{k-t}$.\\
	Moreover, if $n \geq (t + 1)(k − t + 1) + 1$, then $|\mathcal{S}| = \binom{n-t}{k-t}$ holds if and only if $\mathcal{S}$ is the set of all the subsets of size $k$ through a fixed subset  of $\Omega$ of size $t$.
\end{theorem} 
In \cite[Theorem 2.1.2]{wilsonEKR}, Wilson also showed that the upper bound in the previous theorem is sharp. Note that if $t= 1$, then $\mathcal{S}$ is a collection of subsets of size $k$ of an arbitrary set, which are pairwise not disjoint. In the literature, a family of subsets that are pairwise not disjoint, is called an \textit{Erd{\H{o}}s-Ko-Rado set} and the classification of the largest Erd{\H{o}}s-Ko-Rado sets is called the \textit{Erd{\H{o}}s-Ko-Rado problem}, in short \emph{EKR problem}.   Hilton and Milner \cite{hilton_minler} described the largest  Erd{\H{o}}s-Ko-Rado sets $S$ with the property that there is no point contained in all elements of $S$.

This set-theoretical problem can be generalized in a natural way to many other structures such as designs \cite{EKRdesign}, permutation groups \cite{EKRpermutation} and projective geometries. In this article, we work in the projective setting (for an overview, see \cite{surveyEKR}); where this problem is known as the \textit{$q$-analogue} of the Erd{\H{o}}s-Ko-Rado problem.

More precisely, let $q$ be a prime power and let $\PG(n,q)$ be the projective geometry of the subspaces of the vector space $\mathbb{F}_q^{n+1}$ over the finite field $\mathbb{F}_q$. Clearly, results on families of vector spaces pairwise intersecting in at least a  vector space with fixed dimension can be interpreted in projective spaces, and vice versa. Here, in the projective setting, a projective $m$-dimensional subspace of $\PG(n,q)$ will be called $m$-\textit{space}.
In $\PG(n,q)$, we can consider families of $k$-spaces  pairwise intersecting in at least a $t$-dimensional subspace  for $0 \leq t \leq k-1$. In particular for $t=0$, these sets are the Erd{\H{o}}s-Ko-Rado sets of $\PG(n,q)$

Before stating the $q$-analogue of Theorem \ref{thm1}, we briefly recall the definition of $q$\textit{-ary Gaussian coefficient}.

\begin{definition}
	Let $q$ be a prime power,  let $n,k$ be  non-negative integers with $k \leq n $. The $q$\textit{-ary Gaussian
	coefficient} of $n$ and $k$ is defined by
\begin{equation*}
\qbin{n}{k}_q=
\begin{cases} 
\frac{(q^n-1)\cdots(q^{n-k+1}-1)}{(q^k-1)\cdots(q-1)}  \hspace{0.5cm}\textnormal{ if $k>0$} \\
\hspace{1.5cm} 1 \hspace{2.5cm}\textnormal{otherwise} 

\end{cases}
\end{equation*}
\end{definition}
We will write $\qbin{n}{k}$, if the field size $q$ is clear from the context. The number of $k$-spaces in $\PG(n,q)$ is $\qbin{n+1}{k+1}$  and the number of $k$-spaces through a fixed $t$-space  in $\PG(n,q)$, with $0 \leq t \leq k$, is $\qbin{n-t}{k-t}$.
Moreover, we will denote the number $\qbin{n+1}{1}$ by the symbol $\theta_n$.

\begin{theorem}{\textbf{\cite[Theorem 1]{q-analogueEKR}}}\label{q-analogue}
Let $t$ and $k$ be integers, with $0 \leq t \leq k$. Let $\mathcal{S}$ be a set of $k$-spaces
in $\PG(n,q)$, pairwise intersecting in at least a $t$-space.
\begin{itemize}
	\item [(i)] If $n \geq 2k + 1$, then $| \mathcal{S} |
	\leq \qbin{n-t}{k-t}$. Equality holds if and only if $\mathcal{S}$ is the set of all the $k$-spaces, containing a fixed $t$-space of $\PG(n,q)$, or $n = 2k +1$ and $\mathcal{S}$ is the set of all the $k$-spaces in a fixed $(2k -t)$-space.
	\item[(ii)] If $2k-t \leq n \leq 2k$, then $|\mathcal{S}| \leq \qbin{2k-t+1}{k-t}$. Equality holds if and only if $\mathcal{S}$ is the set of all the $k$-spaces in a fixed $(2k-t)$-space.
\end{itemize}
\end{theorem}

\begin{gevolg}\label{corollEKR}
	 Let $\mathcal{S}$ be an Erd\H{o}s-Ko-Rado set of $k$-spaces in $\PG(n, q)$. If  $n \geq  2k + 1$, then
$|\mathcal{S}| \leq \qbin{n}{k}$. Equality holds if and only if $\mathcal{S}$ is the set of all the $k$-spaces, containing a fixed point of $\PG(n,q)$, or $n = 2k +1$ and $\mathcal{S}$ is the set
of all the $k$-spaces in a fixed hyperplane.
\end{gevolg}

Note that in Theorem \ref{q-analogue} the condition $n \geq 2k -t$ is not a restriction, since any two $k$-dimensional subspaces in $\PG(n,q)$, with $n \leq 2k-t$, meet in at least a $t$-dimensional subspace.
Furthermore, as new families of any size can be found by deleting elements, the research is focused on
\textit{maximal} families: these are sets of $k$-spaces pairwise intersecting in at least a $t$-space, not extendable to a larger family of $k$-spaces with the same property. Related to this question, we report the $q$-analogue of the Hilton-Milner result on the second-largest maximal Erd{\H{o}}s-Ko-Rado sets of subspaces in a finite projective space, due to Blokhuis \textit{et al}. In the context of projective spaces, a set of subspaces  through a fixed $t$-space will be called a \emph{$t$-pencil}, and, in particular, a \textit{point-pencil} if $t=0$ and a \textit{line-pencil} if $t=1$.\\

\begin{theorem}{\textbf{\cite[Theorem 1.3, Proposition 3.4]{q-analogue_hilton}}}\label{q-analogue_Hilton}
Let  $\mathcal{S}$ be a maximal set of pairwise intersecting $k$-spaces in $\PG(n,q)$, with $n \geq 2k + 2$, $k \geq 2$ and $q \geq 3$ (or
$n \geq 2k + 4$, $k \geq 2$ and $q = 2$). If $\mathcal{S}$ is not a point-pencil, then
\[ |\mathcal{S} |\leq \qbin{n}{k}-q^{k(k+1)}\qbin{n-k-1}{k}+q^{k+1}.\]
Moreover, if equality holds, then
\begin{itemize}
	\item [(i)] either $\mathcal{S}$ consists of all the $k$-spaces through a fixed point
	P, meeting a fixed $(k+1)$-space $\tau$, with $P \in \tau$, in a $j$-space, $j \geq 1$, and all the $k$-spaces in $\tau$;
	\item [(ii)]or else $k = 2$ and $\mathcal{S}$ is the set of all the planes meeting a fixed plane $\pi$ in at
	least a line.
\end{itemize}
\end{theorem}
The Erd{\H{o}}s-Ko-Rado problem for $k = 1$ has been solved completely. Indeed, in $\PG(n,q)$ with $n \geq 3$, a maximal Erd{\H{o}}s-Ko-Rado set of lines is either the set of all the lines through a fixed point or the set of all the lines contained in a fixed plane. It is possible to generalize this result for a maximal family $\mathcal{S}$ of $k$-spaces, pairwise intersecting in a $(k-1)$-space, in a projective space $\PG(n,q)$, $n \geq k +2$. Precisely
\begin{theorem}{\textbf{\cite[Section 9.3]{brouwer}}}\label{basisEKRthm}
Let $\mathcal{S}$ be a set of projective $k$-spaces, pairwise intersecting in a $(k-1)$-space in $\PG(n,q)$, $n\geq k+2$, then all the $k$-spaces of $\mathcal{S}$ go through a fixed $(k-1)$-space or they are contained in a fixed $(k + 1)$-space.
\end{theorem}
The Erd{\H{o}}s-Ko-Rado problem for sets of projective planes is trivial if $n \leq 4$.  For $n = 5$, Blokhuis, Brouwer and Sz\H{o}nyi  classified the six largest examples \cite[Section 6]{kneser}.\\
De Boeck investigated the maximal Erd\H{o}s-Ko-Rado sets of planes in $\PG(n, q)$ with $n  \geq 5$, see \cite{EKRplanes}. He characterized those sets with sufficiently large size and showed that they belong to one of the 11 known examples, explicitly described in his work.\\


In \cite{precisek-2,(k-2)-spaces}, a classification of the largest examples of sets of $k$-spaces in PG$(n,q)$ pairwise intersecting in precisely a $(k-2)$-space is given. In \cite{012clique}, Brouwer and Hemmeter investigated sets of generators, pairwise intersecting in at least a space with codimension 2, in
quadrics and symplectic polar spaces. 
In this paper, we will study the projective analogue of this question. Let $f(k,q)=\max\{3q^4+6q^3+5q^2+q+1,\theta_{k+1}+q^4+2q^3+3q^2\}$ and so \begin{align*}
    f(k,q)=\begin{cases}
    3q^4+6q^3+5q^2+q+1 \ \text{    if } k=3, q\geq 2 \text{    or } k=4, q=2 \\
    \theta_{k+1}+q^4+2q^3+3q^2 \ \ \text{    else. }
    \end{cases}
\end{align*} We analyze the sets of $k$-spaces in $\PG(n,q)$ pairwise intersecting in at least a $(k-2)$-space and with more than $f(k,q)$ elements.  We will suppose that these sets $\mathcal{S}$ of subspaces are maximal. During the discussion, we will give bounds on the size of the largest examples and we will indicate the order of such families of $k$-spaces in $\PG(n,q)$, using
the \textit{big $O$ notation}.\\

In \cite{ellis}, families of subspaces pairwise intersecting in at least a $t$-space were investigated. More specifically, the author investigates the largest non-trivial example of a set of $k$-spaces, pairwise intersecting in at least a $t$-space in $\PG(n,q)$. The main theorem of this article is consistent with Theorem $1$ from \cite{ellis}.

\begin{theorem}{\textbf{\cite[Theorem 1]{ellis}}}
Let $\mathcal{F}$ be a 
set of $k$-spaces pairwise intersecting in at least a $t$-space in $\PG(n,q)$, $n \geq 2k+5+ \frac{t(t+5)}{2}$, of maximum size, with $\mathcal{F}$ not a $t$-pencil, then $\mathcal{F}$ is one of the following examples:
\begin{itemize}
    \item [$i)$] the set of $k$-spaces, meeting a fixed $(t+2)$-space in at least a $(t+1)$-space,
    \item [$ii)$]the set of $k$-spaces in a fixed $(k+1)$-space $Y$ together with the set of $k$-spaces through a $t$-space $\pi\subset Y$, that have at least a $(t+1)$-space in common with $Y$.
\end{itemize}
\end{theorem}
Note that the two examples in the previous theorem correspond to Example \ref{voorbeeldenalgemeen}$(ii)$ and $(iii)$  for $t=k-2$ respectively. 
While, in [10], David Ellis classifies the largest non-trivial example for all values of $t$, here we specify this problem classifying the ten largest examples when $t=k-2$, see Main Theorem \ref{overzicht2}.

We end this section with some examples of maximal sets $\mathcal{S}$ of $k$-spaces in $\PG(n,q)$ pairwise intersecting in at least a $(k-2)$-space, $n \geq k+2$ and $k \geq 3$. We add a proof of maximality for the examples for which it is not straightforward.

\begin{example}\label{voorbeeldenalgemeen}

\begin{itemize}
      \item [$(i)$] \emph{$(k-2)$-pencil:} the set $\mathcal{S}$ is the set of all $k$-spaces that contain  a fixed $(k-2)$-space.  Then $|\mathcal{S}|=\qbin{n-k+2}{2}$.
    \item [$(ii)$] \emph{Star:} there exists a $k$-space $\zeta$ such that $\mathcal{S}$ contains all $k$-spaces that have at least a $(k-1)$-space in common with $\zeta$. Then $|\mathcal{S}|=q \theta_k \theta_{n-k-1}+1$.
    \item[$(iii)$] \emph{Generalized Hilton-Milner example:} there exists a $(k+1)$-space $\nu$ and a $(k-2)$-space $\pi\subset \nu$ such that $\mathcal{S}$ consists of all $k$-spaces in $\nu$ (type 1), together with all $k$-spaces of $\PG(n,q)$, not in $\nu$, through $\pi$ that intersect $\nu$ in a $(k-1)$-space (type 2). Then $|\mathcal{S}|=\theta_{k+1}+ q^2(q^2+q+1) \theta_{n-k-2}$.

    \item[$(iv)$] There exists a $(k+2)$-space $\rho$, a $k$-space $\alpha \subset \rho$ and a $(k-2)$-space $\pi\subset \alpha$ so that $\mathcal{S}$ contains all $k$-spaces in $\rho$ that meet $\alpha$ in  a $(k-1)$-space not through $\pi$ (type 1),  all $k$-spaces in $\rho$ through $\pi$ (type 2), and all $k$-spaces in $\PG(n,q)$, not in $\rho$, that contain a $(k-1)$-space of $\alpha$ through $\pi$ (type 3). Then $|\mathcal{S}|=(q+1)\theta_{n-k}+q^3(q+1)\theta_{k-2}+q^4-q$.

\textnormal{This example will be discussed in Proposition 2.5.}

\begin{figure}[h!] 
\centering
\definecolor{qqwuqq}{rgb}{0,0.39215686274509803,0}
\definecolor{qqqqff}{rgb}{0,0,1}
\definecolor{ffqqqq}{rgb}{1,0,0}
\begin{tikzpicture}[line cap=round,line join=round,>=triangle 45,x=1cm,y=1cm,scale=0.33]
\clip(-12.110416907277649,-7.994359453778139) rectangle (11.665742069618393,9.432875037725609);
\draw [rotate around={41.835971459820975:(0.7739561172384822,-0.7261943713711371)},line width=0.8pt] (0.7739561172384822,-0.7261943713711371) ellipse (7.515681694324461cm and 3.928027465508016cm);
\draw [line width=0.8pt] (2.64,0.92) circle (2.1661549101291957cm);
\draw [line width=0.8pt] (2.45,0.13) circle (1.1244380174566182cm);
\draw [rotate around={61.19796975948429:(1.693478646313929,-1.5636335417967262)},line width=0.8pt,dash pattern=on 2pt off 3pt,color=ffqqqq] (1.693478646313929,-1.5636335417967262) ellipse (3.6789125397853346cm and 1.7950249650548737cm);
\draw [rotate around={36.656108415966884:(0.8077953790001363,-1.1263974984392835)},line width=0.8pt,dash pattern=on 2pt off 3pt,color=ffqqqq] (0.8077953790001363,-1.1263974984392835) ellipse (3.5192555909516052cm and 1.8323303997725884cm);
\draw [rotate around={7.6744795168939195:(0.06785745947215666,0.18531063163304232)},line width=0.8pt,dash pattern=on 2pt off 3pt,color=qqqqff] (0.06785745947215666,0.18531063163304232) ellipse (2.438938473400423cm and 1.8450453518989294cm);
\draw [rotate around={40.49194625346984:(3.9794532195605288,2.2084019527121024)},line width=0.8pt,dash pattern=on 2pt off 3pt,color=qqqqff] (3.9794532195605288,2.2084019527121024) ellipse (2.3934159430050927cm and 1.21445240062208cm);
\draw [rotate around={-41.8361163811192:(3.7522489926541422,-1.1182853553527468)},line width=0.8pt,dash pattern=on 2pt off 3pt,color=qqwuqq] (3.7522489926541422,-1.1182853553527468) ellipse (3.3921314712323936cm and 2.745190681796649cm);
\draw [rotate around={-77.4034267641815:(1.4193006863489073,3.802114708854659)},line width=0.8pt,dash pattern=on 2pt off 3pt,color=qqwuqq] (1.4193006863489073,3.802114708854659) ellipse (5.696714538512506cm and 2.3561372340601765cm);
\begin{scriptsize}
\draw[color=black] (-4.101180727785475,1.3681230214944122) node {\normalsize{$\rho$}};
\draw[color=black] (1.2288763594784877,3.2052652590350867) node {\normalsize{$\alpha$}};
\draw[color=black] (2.7916156101946235,-0.16671017117321907) node {\normalsize{$\pi$}};
\end{scriptsize}
\end{tikzpicture}  \caption{Example $(iv)$: The blue, red and green $k$-spaces correspond to the $k$-spaces of type $1$, $2$ and $3$, respectively.}
\end{figure}
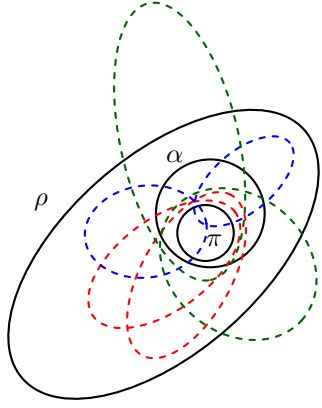

  \newpage
  
\begin{lemma}
    The set $\mathcal{S}$ is maximal.
\end{lemma}
\begin{proof}
 \textnormal{ Suppose there is a $k$-space $E\notin \mathcal{S}$, meeting all elements of $\mathcal{S}$ in at least a $(k-2)$-space. If $\pi \not \subset E$, then $E$ contains a $(k-1)$-space $\sigma_E \subset \alpha$. This follows since $E$ meets all $k$-spaces in $\mathcal{S}$ of type 3, in at least a $(k-2)$-space. Let $G$ be an element of $\mathcal{S}$ of type 2 such that
 $\langle G , \alpha \rangle=\rho$, and so $G\cap\alpha=\pi$. We have
    \begin{equation*}
    \dim(E\cap \rho)\geq \dim(E\cap \alpha)+\dim(E\cap G)-\dim(E\cap G \cap \alpha)\geq (k-1)+(k-2)-(k-3)\geq k.
    \end{equation*} 
    So, $E\subset \rho$, which implies that $E\in \mathcal{S}$ (type 1), a contradiction. Now, we suppose that $\pi\subset E$. Let $F_1$ and $F_2$ be two elements of $\mathcal{S}$ of type $1$, with $\langle F_1,F_2 \rangle =\rho$ and  $\dim(\pi \cap F_1 \cap F_2)=k-4$. {First we show that their existence is assured. Indeed, let $\pi_1$ and $\pi_2$ be two different $(k-3)$-spaces in $\pi$ and let $\alpha_i$ be a $(k-1)$-space in $\alpha$ through $\pi_i$, $i=1,2$. Consider $P_1$ be a point in $\rho\setminus \alpha$ and let  $F_1=\langle P_1,\alpha_1\rangle $. Finally, consider $P_2$ be a point in $\rho \setminus \langle \alpha, F_1\rangle $ and let $F_2=\langle P_2,\alpha_2\rangle $.}
    Since  $E \not \in \mathcal{S}$ and $\pi \subset E$, we know that $E$ cannot contain a $(k-1)$-space of $\alpha$, and so, $E\cap \alpha=\pi$. Hence, from $F_1\cap F_2 \subset \alpha$, there follows that   $\dim(E \cap F_1 \cap F_2)=\dim(\pi \cap F_1 \cap F_2)$. Then
    \begin{equation*}
\begin{split}
\dim(E \cap \rho)=\dim(E\cap \langle F_1, F_2 \rangle)\geq
\dim(E\cap F_1)+\dim(E\cap F_2)& \\-\dim(E\cap F_1 \cap F_2)\geq (k-2)+(k-2)-(k-4)\geq k
\end{split}
    \end{equation*}
Hence, $E\subset \rho$ which implies that $E\in \mathcal{S}$, type 2, again a contradiction.}
    \end{proof}

    \item[$(v)$] There is a $(k+2)$-space $\rho$, and a $(k-1)$-space $\alpha \subset \rho$  such that $\mathcal{S}$ contains all $k$-spaces in $\rho$ that meet $\alpha$ in at least a $(k-2)$-space (type 1), and all $k$-spaces in $\PG(n,q)$, not in $\rho$, through $\alpha$ (type 2). Note that all $k$-spaces in $\PG(n,q)$ through $\alpha$ are contained in $\mathcal{S}$.
    Then $|\mathcal{S}|=\theta_{n-k}+ q^2 (q^2+q+1) \theta_{k-1}$.

\textnormal{This example will be discussed in Proposition \ref{prop5}.}

  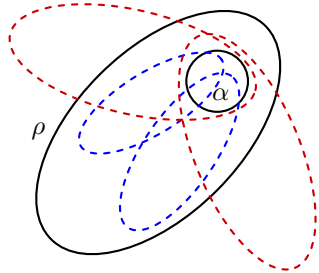
\begin{figure}[h!]
    \centering
\definecolor{ccqqqq}{rgb}{0.8,0,0}
\definecolor{qqqqff}{rgb}{0,0,1}
\begin{tikzpicture}[line cap=round,line join=round,>=triangle 45,x=1cm,y=1cm,scale=0.38]
\clip(-4.17,-5.99) rectangle (18.23,4.93);
\draw [rotate around={45.27590049754996:(7.72,-0.35)},line width=0.8pt] (7.72,-0.35) ellipse (5.254750324487912cm and 2.864524563117887cm);
\draw [line width=0.8pt] (9.76,1.44) circle (1.0667708282475676cm);
\draw [rotate around={56.093723011557856:(8.46,-1.02)},line width=0.8pt,dash pattern=on 2pt off 3pt,color=qqqqff] (8.46,-1.02) ellipse (3.195468857636123cm and 1.2521266789435819cm);
\draw [rotate around={31.776932347298654:(7.48,0.68)},line width=0.8pt,dash pattern=on 2pt off 3pt,color=qqqqff] (7.48,0.68) ellipse (2.8668214711795144cm and 1.0728771353774806cm);
\draw [rotate around={-67.82376365850719:(10.8,-1.02)},line width=0.8pt,dash pattern=on 2pt off 3pt,color=ccqqqq] (10.8,-1.02) ellipse (4.375072333115626cm and 1.8299885026943925cm);
\draw [rotate around={-12.983775584947859:(6.84,2.11)},line width=0.8pt,dash pattern=on 2pt off 3pt,color=ccqqqq] (6.84,2.11) ellipse (4.35598972476016cm and 1.811890306341988cm);
\begin{scriptsize}
\draw[color=black] (3.57,-0.3) node {\normalsize{$\rho$}};
\draw[color=black] (9.89,1) node {\normalsize{$\alpha$}};
\end{scriptsize}
\end{tikzpicture}
\caption{Example$(v)$: The blue and red $k$-spaces correspond to the $k$-spaces of type $1$, $2$, respectively.}
\end{figure}

\begin{lemma}
    The set $\mathcal{S}$ is maximal.
\end{lemma}
\begin{proof}
   \textnormal{Suppose there is a $k$-space $E\notin \mathcal{S}$, meeting all elements of $\mathcal{S}$ in at least a $(k-2)$-space. Then $E$ contains a $(k-2)$-space $\sigma_E$ in $\alpha$, since $E$ meets all elements of $\mathcal{S}$ of type 2. Note that $E$ cannot contain $\alpha$, since then, $E$ would be a $k$-space in $\mathcal{S}$. Let $\sigma_1$ and $\sigma_2$ be two distinct $(k-2)$-spaces in $\alpha$ with $\sigma_1,\sigma_2 \neq \sigma_E$ and $\dim(\sigma_1\cap\sigma_2\cap \sigma_E)=k-4$. Consider $F_1$ and $F_2$, two elements of $\mathcal{S}$ of type 1 through $\sigma_1$ and $\sigma_2$, respectively, with $\dim(F_1 \cap F_2)=k-2$. Note that $\dim(E \cap F_1 \cap F_2)=k-4$. Indeed,} 
    $$ k-4 \leq \dim(E\cap F_1 \cap F_2) \leq k-2.$$
    \textnormal{\begin{itemize}
    \item [$(a)$] If $\dim(E \cap F_1 \cap F_2)=k-2$, then $E \cap F_1 \cap F_2 \cap \alpha=F_1 \cap F_2 \cap \alpha$, a contradiction.
    \item [$(b)$] If $\dim(E \cap F_1 \cap F_2)=k-3$, there exists a point $P \in F_1 \cap F_2 \cap E$ not in $\alpha$ and $\dim(E \cap \rho) \geq k-1$. Since $E \not \in \mathcal{S}$, then $E \not \subset \rho$. The only possibility is $\dim(E \cap \rho)=k-1$, but then we can find a $k$-space  $F$ of type 1 such that $E\cap F$ is a $(k-3)$-space, again a contradiction.
    \end{itemize}}
\textnormal{Hence, $\dim(E \cap F_1 \cap F_2)=k-4$ and
    \begin{equation*}
    \begin{split}
     \dim(E\cap \rho)=\dim(E\cap \langle F_1, F_2 \rangle)&\geq \dim(E\cap F_1)+\dim(E\cap F_2)\\
     &-\dim(E\cap F_1 \cap F_2)\geq (k-2)+(k-2)-(k-4)\geq k.
     \end{split}
    \end{equation*} And so, $E\subset \rho$, which implies that $E\in \mathcal{S}$, a contradiction.}
    \end{proof}

\item [$(vi)$] There are two $(k+2)$-spaces $\rho_1,\rho_2$ intersecting in a $(k+1)$-space $\alpha=\rho_1\cap \rho_2$. There are two $(k-1)$-spaces $\pi_A,\pi_B\subset \alpha$ with $\pi_A\cap \pi_B$ the $(k-2)$-space $\lambda$, there is a point $P_{AB} \in \alpha\setminus \langle \pi_A,\pi_B \rangle$, and let $\lambda_A, \lambda_B \subset \lambda$ be two different $(k-3)$-spaces. Then $\mathcal{S}$ contains 
    \begin{enumerate}
        \item[type 1.] all $k$-spaces in $\alpha$,
        \item[type 2.]  all $k$-spaces of $\PG(n,q)$ through $\langle P_{AB},\lambda \rangle$,  not in $\rho_1$ and not in $\rho_2$.
        \item[type 3.]  all $k$-spaces in $\rho_1$, not in $\alpha$, through $P_{AB}$ and a $(k-2)$-space in $\pi_A$ through $\lambda_A$,
        \item[type 4.]  all $k$-spaces in $\rho_1$, not in $\alpha$, through $P_{AB}$ and a $(k-2)$-space in $\pi_B$ through $\lambda_B$,
        \item[type 5.]  all $k$-spaces in $\rho_2$, not in $\alpha$, through $P_{AB}$ and a $(k-2)$-space in $\pi_A$ through $\lambda_B$,
        \item[type 6.]  all $k$-spaces in $\rho_2$, not in $\alpha$, through $P_{AB}$ and a $(k-2)$-space in $\pi_B$ through $\lambda_A$.
    \end{enumerate}
Then $|\mathcal{S}|= \theta_{n-k}+q^2 \theta_{k-1}+4q^3$.\\

\textnormal{This example will be discussed in Proposition \ref{prop6k=3} for $k=3$ and in Proposition \ref{prop6kmeerdan3} for $k>3$.}

\begin{figure}[h!]
    \centering
\definecolor{ccqqqq}{rgb}{0.8,0,0}
\definecolor{qqqqff}{rgb}{0,0,1}
\definecolor{qqwuqq}{rgb}{0,0.39215686274509803,0}
	\definecolor{ffzztt}{rgb}{1,0.6,0.2}
\begin{tikzpicture}[line cap=round,line join=round,>=triangle 45,x=1cm,y=1cm,scale=0.51]
\clip(-13.037950890260937,-6.351958797718809) rectangle (11.963335364369781,6.938077890838478);
\draw [rotate around={-53.06917102705861:(2.73,0.25)},line width=0.8pt] (2.73,0.25) ellipse (6.529021470134076cm and 3.285714740733243cm);
\draw [rotate around={55.813877513624966:(-2.67,0.14)},line width=0.8pt] (-2.67,0.14) ellipse (6.543782930999828cm and 3.0936378340143063cm);
\draw [line width=0.8pt] (-0.73,3.15) circle (1.2425779653607254cm);
\draw [line width=0.8pt] (0.5120078446767284,3.1876366013538404) circle (1.2425779653607254cm);
\draw [rotate around={-89.66903576176806:(-0.11571359371648168,-0.598299934225342)},line width=0.8pt,dash pattern=on 2pt off 3pt,color=qqwuqq] (-0.11571359371648168,-0.598299934225342) ellipse (4.627035258195526cm and 1.2623662400486673cm);
\draw [line width=0.8pt] (-0.10374059835660195,3.41661695378231) circle (0.5774467420658767cm);
\draw [line width=0.8pt] (-0.10374059835660195,2.860403342673657) circle (0.5562136111086531cm);
\draw [rotate around={70.07441934181567:(-1.1367087332726713,0.039605743479772915)},line width=0.8pt,dash pattern=on 2pt off 3pt,color=qqqqff] (-1.1367087332726713,0.039605743479772915) ellipse (3.58024255995664cm and 1.1784214387887235cm);
\draw [rotate around={41.814945181755704:(-2.442614054139102,1.2877675641923922)},line width=0.8pt,dash pattern=on 2pt off 3pt,color=ccqqqq] (-2.442614054139102,1.2877675641923922) ellipse (3.745744489564882cm and 1.5021645972555848cm);
\draw [rotate around={-62.43812695799701:(1.4084001571710807,0.04939030228868189)},line width=0.8pt,dash pattern=on 2pt off 3pt,color=qqqqff] (1.4084001571710807,0.04939030228868189) ellipse (3.756230277271819cm and 1.509070780483217cm);
\draw [rotate around={-38.56195048500158:(2.5382740544095,1.3109511402995522)},line width=0.8pt,dash pattern=on 2pt off 3pt,color=ccqqqq] (2.5382740544095,1.3109511402995522) ellipse (3.9753805280474555cm and 1.614294582027233cm);
\draw [line width=0.8pt,dash pattern=on 2pt off 3pt, color=ffzztt] (-0.14347014200722,4.64823280695147) circle (1.0398216662005848cm);
\begin{scriptsize}
\draw[color=black] (6.99,0.5643892454504077) node {\normalsize{$\rho_2$}};
\draw[color=black] (-6.8747418222639896,0.5643892454504077) node {\normalsize{$\rho_1$}};
\draw[color=black] (-1.5391014630440187,3.199542239187616) node {\normalsize{$\pi_A$}};
\draw [fill=black] (-0.07948249109653749,1.7130646142144335) circle (2.5pt);
\draw[color=black] (-0.11352340855708789,0.97542239187616) node {\normalsize{$P_{AB}$}};
\draw[color=black] (1.2702765220146597,3.5479922218305524) node {\normalsize{$\pi_B$}};
\draw[color=black] (-0.12352340855708789,6.332885956643122) node {\normalsize{$\alpha$}};
\draw[color=black] (-0.12352340855708789,3.732885956643122) node {\normalsize{$\lambda_A$}};
\draw[color=black] (-0.10174528464190434,2.729314132629496) node {\normalsize{$\lambda_B$}};
\end{scriptsize}
\end{tikzpicture}\caption{Example$(vi)$: The orange $k$-space is of type 1, the green one of type 2, the red ones of type 3 and 6, and  the blue ones of type 4 and 5.}
\end{figure}
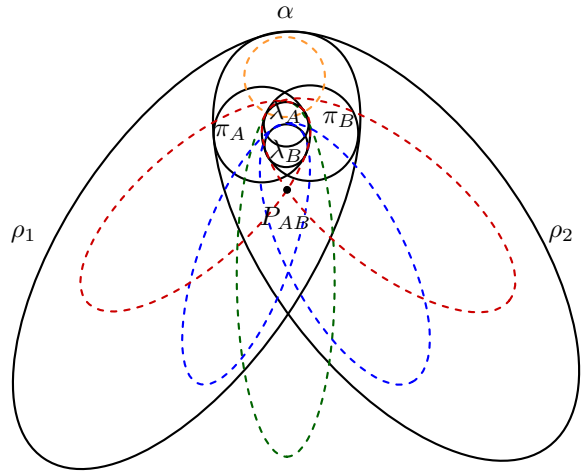

\begin{lemma}
    The set $\mathcal{S}$  is maximal.
    \end{lemma}
    \begin{proof}
\textnormal{    Suppose there is a $k$-space $E\notin \mathcal{S}$, meeting all elements of $\mathcal{S}$ in at least a $(k-2)$-space. Suppose first that $P_{AB}\notin E$. As $E$ contains at least a $(k-2)$-space of all elements of $\mathcal{S}$, type 1 and 2, $E$ contains a $(k-1)$-space $\beta$ in $\alpha$ such that $\beta$ contains a $(k-2)$-space of $\langle P_{AB}, \lambda \rangle$, not through $P_{AB}$. Consider now the elements $F$ and $G$ of $\mathcal{S}$, type $3$ and $4$ respectively, with $F\cap G\cap \alpha=\langle P_{AB}, \lambda_A\cap \lambda_B\rangle$. If $E\not \subset \rho_1$, then $\dim(E \cap F \cap G)\leq k-4$ and
    \begin{equation*}
    \begin{split}
 k-1&=\dim(E\cap \alpha)= \dim(E\cap \rho_1)= \dim(E\cap \langle F, G\rangle)\geq\\
 & \dim(E\cap F)+\dim(E\cap G)-\dim(E\cap F\cap G)\geq (k-2)+(k-2)-(k-4)\geq k,
    \end{split}
    \end{equation*}
a contradiction. Hence, $E\subset \rho_1$. Analogously, we find that $E \subset \rho_2$, using two elements of $\mathcal{S}$ of type $5$ and $6$. And so, $E\subset \rho_1\cap \rho_2=\alpha$, which implies that $E\in \mathcal{S}$, type 1, a contradiction. So now we can suppose that $P_{AB}\in E$. Then $E$ contains a $(k-1)$-space of $\alpha$ that meets $\lambda$ in a $(k-3)$-space. This follows
since $E$ meets the elements of $\mathcal{S}$ of type $1$ and $2$ in at least a $(k-2)$-space.
Note that the dimension of $E\cap \pi_A$ and $E\cap \pi_B$ is $k-2$ or $k-3$ as $E\cap \lambda$ is a $(k-3)$-space. Moreover, the latter spaces do not both have the same dimension. Indeed, if $\dim(E \cap \pi_A)=\dim(E \cap \pi_B)=k-2$, then $E \subset \alpha$, type 1, a contradiction. 
Moreover, since $E$ contains $P_{AB}$, and since $\dim(E\cap \alpha)=k-1$, we know that  $\dim(E \cap \langle \pi_A, \pi_B \rangle) =k-2$. If $\dim(E \cap \pi_A)=\dim(E \cap \pi_B)=k-3$, then $E \cap \pi_A=E \cap \lambda=E \cap \pi_B=E \cap \langle \pi_A, \pi_B \rangle$, which cannot occur.\\
By a similar argument, we find that the dimension of $E\cap \lambda_A$ and $E\cap \lambda_B$ is $k-3$ or $k-4$, both not the same dimension. Then $E$ contains a $(k-2)$-space of $\pi_A$ or $\pi_B$, and $E$ contains $\lambda_A$ or $\lambda_B$. W.l.o.g. we can suppose that $E$ contains $\lambda_A$ and a $(k-2)$-space of $\pi_A$, and meets $\pi_B$ in $\lambda_A$.\\
Let $H$ be an element of type $1$ of $\mathcal{S}$, and let $G$ be an element of type $4$ of $\mathcal{S}$ through a $(k-2)$-space $\sigma \neq \lambda$ in $\pi_B$ with $H \cap G=\sigma$. Then, since $\dim(E \cap G \cap H)=k-4$, 
\begin{equation*}
\begin{split}
    \dim(E\cap \rho_1)=&\dim(E\cap \langle G,H \rangle)
\geq \dim(E\cap G)+\dim(E\cap H)\\
    &-\dim(E\cap G\cap H)\geq (k-2)+(k-2)-(k-4)\geq k,
\end{split}
\end{equation*}
and so $E\subset \rho_1$. Hence, $E\in \mathcal{S}$, type $3$, a contradiction.}
    \end{proof}

   \item[$(vii)$] There is a $(k-3)$-space $\gamma$ contained in all $k$-spaces of $\mathcal{S}$. In the quotient space  $\PG(n,q)/ \gamma$, the set of planes corresponding to the elements of $\mathcal{S}$ is the set of planes of example $VIII$ in \cite{EKRplanes}: Let $\Psi$ be an $(n-k+2)$-space, disjoint to $\gamma$, in $\PG(n,q)$. Consider two solids $\sigma_1$ and $\sigma_2$ in $\Psi$, intersecting in a line $l$. Take the points $P_1$ and $P_2$ on $l$.  Then $\mathcal{S}$ is the set containing all $k$-spaces through $\langle \gamma,l\rangle$ (type 1), all $k$-spaces through $\langle \gamma, P_1 \rangle$ that contain a line in $\sigma_1$ and a line in $\sigma_2$ skew to $\gamma$ (type 2), and all $k$-spaces through $\langle \gamma, P_2\rangle$ in $\langle \gamma,\sigma_1\rangle$ or in $\langle \gamma,\sigma_2\rangle$ (type 3).
    Then $|\mathcal{S}|=\theta_{n-k}+ q^4+2q^3+3q^2$.\\
    
\textnormal{In Lemma \ref{maxmaarten8}, we  prove that the set $\mathcal{S}$ is maximal.}

\begin{figure}[h!]
    \centering
\definecolor{ffqqqq}{rgb}{1,0,0}
\definecolor{qqwuqq}{rgb}{0,0.39215686274509803,0}
\definecolor{sexdts}{rgb}{0.1803921568627451,0.49019607843137253,0.19607843137254902}
\definecolor{rvwvcq}{rgb}{0.08235294117647059,0.396078431372549,0.7529411764705882}
\begin{tikzpicture}[line cap=round,line join=round,>=triangle 45,x=1cm,y=1cm,scale=0.78]
\clip(-7.925724667883978,-3.063460824240777) rectangle (8.849845539882477,3);
\fill[line width=0.8pt,color=rvwvcq,fill=rvwvcq,fill opacity=0.05] (0,2.418429242254347) -- (2.1785816316471482,-1.6492751461457986) -- (-2.197707227597139,-1.7053814135720076) -- cycle;
\fill[line width=0.8pt,color=ffqqqq,fill=ffqqqq,fill opacity=0.05] (0,3) -- (-1.4614876107779,-3.1682507742901) -- (0,1) -- cycle;
\fill[line width=0.8pt,color=ffqqqq,fill=ffqqqq,fill opacity=0.05] (1.4614876107779,-3.1682507742901) -- (0,3) -- (0,1) -- cycle;
\fill[line width=0.8pt,color=qqwuqq,fill=qqwuqq,fill opacity=0.04] (0,1) -- (2.3202546646618,0.3260427553064) -- (3.1197964044848,-1.6876179227661) -- cycle;
\fill[line width=0.8pt,color=qqwuqq,fill=qqwuqq,fill opacity=0.04] (0,1) -- (-3.1197964044848,-1.6876179227661) -- (-2.3202546646618,0.3260427553) -- cycle;
\draw [shift={(7.759315031590855,-6.753068855411385)},line width=0.8pt]  plot[domain=2.2428343529918475:2.719276510038199,variable=\t]({1*12.463118464407989*cos(\t r)+0*12.463118464407989*sin(\t r)},{0*12.463118464407989*cos(\t r)+1*12.463118464407989*sin(\t r)});
\draw [shift={(-2.789442594519279,-2.0131851503265046)},line width=0.8pt]  plot[domain=2.7190357130443834:5.602432282667888,variable=\t]({1*0.8983969224547149*cos(\t r)+0*0.8983969224547149*sin(\t r)},{0*0.8983969224547149*cos(\t r)+1*0.8983969224547149*sin(\t r)});
\draw [shift={(-13.632015596660036,6.566007798330017)},line width=0.8pt]  plot[domain=5.6131084485081635:5.895540788551897,variable=\t]({1*14.724547260905952*cos(\t r)+0*14.724547260905952*sin(\t r)},{0*14.724547260905952*cos(\t r)+1*14.724547260905952*sin(\t r)});
\draw [shift={(-7.759315031590855,-6.753068855411385)},line width=0.8pt]  plot[domain=0.4223161435515943:0.8987583005979458,variable=\t]({1*12.463118464407989*cos(\t r)+0*12.463118464407989*sin(\t r)},{0*12.463118464407989*cos(\t r)+1*12.463118464407989*sin(\t r)});
\draw [shift={(13.632015596660036,6.566007798330017)},line width=0.8pt]  plot[domain=3.5292371722174822:3.811669512261216,variable=\t]({1*14.724547260905952*cos(\t r)+0*14.724547260905952*sin(\t r)},{0*14.724547260905952*cos(\t r)+1*14.724547260905952*sin(\t r)});
\draw [shift={(2.789442594519279,-2.0131851503265046)},line width=0.8pt]  plot[domain=-2.4608396290780945:0.42255694054541004,variable=\t]({1*0.8983969224547149*cos(\t r)+0*0.8983969224547149*sin(\t r)},{0*0.8983969224547149*cos(\t r)+1*0.8983969224547149*sin(\t r)});
\draw [line width=0.8pt] (0,3)-- (0,1);
\draw [line width=0.8pt,color=rvwvcq] (0,2.418429242254347)-- (2.1785816316471482,-1.6492751461457986);
\draw [line width=0.8pt,color=rvwvcq] (-2.197707227597139,-1.7053814135720076)-- (0,2.418429242254347);
\draw [line width=0.8pt,color=qqwuqq] (0,1)-- (3.1197964044848,-1.6876179227661);
\draw [shift={(-0.028876364862109104,-0.170870062664042)},line width=0.8pt,color=rvwvcq,fill=rvwvcq,fill opacity=0.05]  plot[domain=3.757353975775618:5.693063605939563,variable=\t]({1*2.6567936309049984*cos(\t r)+0*2.6567936309049984*sin(\t r)},{0*2.6567936309049984*cos(\t r)+1*2.6567936309049984*sin(\t r)});
\draw [line width=0.8pt,color=rvwvcq] (0,2.418429242254347)-- (2.1785816316471482,-1.6492751461457986);
\draw [line width=0.8pt,color=rvwvcq] (-2.197707227597139,-1.7053814135720076)-- (0,2.418429242254347);
\draw [line width=0.8pt,color=ffqqqq] (0,3)-- (-1.4614876107779,-3.1682507742901);
\draw [line width=0.8pt,color=ffqqqq] (-1.4614876107779,-3.1682507742901)-- (0,1);
\draw [line width=0.8pt,color=ffqqqq] (1.4614876107779,-3.1682507742901)-- (0,3);
\draw [line width=0.8pt,color=ffqqqq] (0,1)-- (1.4614876107779,-3.1682507742901);
\draw [line width=0.8pt,color=qqwuqq] (0,1)-- (2.3202546646618,0.3260427553064);
\draw [line width=0.8pt,color=qqwuqq] (2.3202546646618,0.3260427553064)-- (3.1197964044848,-1.6876179227661);
\draw [line width=0.8pt,color=qqwuqq] (3.1197964044848,-1.6876179227661)-- (0,1);
\draw [line width=0.8pt,color=qqwuqq] (0,1)-- (-3.1197964044848,-1.6876179227661);
\draw [line width=0.8pt,color=qqwuqq] (-3.1197964044848,-1.6876179227661)-- (-2.3202546646618,0.3260427553);
\draw [line width=0.8pt,color=qqwuqq] (-2.3202546646618,0.3260427553)-- (0,1);
\begin{scriptsize}
\draw[color=black] (-1.7959046625747865,1.8733040851489817) node {\normalsize{$\sigma_1$}};
\draw[color=black] (2.186997708024663,1.5771775148441896) node {\normalsize{$\sigma_2$}};
\draw[color=black] (0.06969273034540183,3.398355922218662) node {\normalsize{$l$}};
\draw [fill=sexdts] (0,1) circle (1pt);
\draw[color=black] (0.3,2.35) node {\normalsize{$P_1$}};
\draw[color=black] (0.15853070143683937,1.2810509445393972) node {\normalsize{$P_2$}};;
\end{scriptsize}
\end{tikzpicture}
    \caption{Example$(vii)$: The red, blue and green planes correspond to the $k$-spaces of type $1$, $2$ and $3$ in $\PG(n,q)/\gamma$, respectively.}

\end{figure}
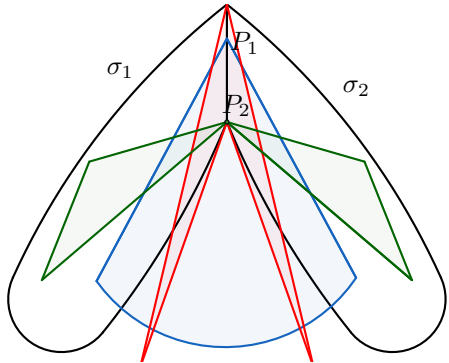

\item[$(viii)$]  There is a $(k-3)$-space $\gamma$ contained in all $k$-spaces of $\mathcal{S}$. In the quotient space  $\PG(n,q)/ \gamma$, the set of planes corresponding to the elements of $\mathcal{S}$ is the set of planes of example $IX$ in \cite{EKRplanes}: Let $\Psi$ be an $(n-k+2)$-space, disjoint to $\gamma$, in $\PG(n,q)$, and let $l$ be a line and $\sigma$ a solid skew to $l$, both in $\Psi$. Denote $\langle l, \sigma \rangle$ by  $\rho$. Let $P_1$ and $P_2$ be two points on $l$ and let $\mathcal{R}_1$ and $\mathcal{R}_2$ be a regulus and its opposite regulus in $\sigma$.
Then $\mathcal{S}$ is the set containing all $k$-spaces through $\langle \gamma,l\rangle$ (type 1), all $k$-spaces through $\langle \gamma, P_1 \rangle $ in the $(k+1)$-space generated by $\gamma,l$ and a fixed line of $\mathcal{R}_1$ (type 2), and all $k$-spaces through $\langle \gamma, P_2 \rangle$ in the $(k+1)$-space generated by $\gamma,l$ and a fixed line of $\mathcal{R}_2$ (type 3).
Then $|\mathcal{S}|=\theta_{n-k}+ 2q^3+2q^2$.\\

\textnormal{In Lemma \ref{maxmaarten9}, we prove that the set $\mathcal{S}$ is maximal.}

\begin{figure}[h!]
    \centering
    \definecolor{qqwuqq}{rgb}{0,0.39215686274509803,0}
\definecolor{rvwvcq}{rgb}{0.08235294117647059,0.396078431372549,0.7529411764705882}
\definecolor{ccqqqq}{rgb}{0.8,0,0}
\begin{tikzpicture}[line cap=round,line join=round,>=triangle 45,x=1cm,y=1cm,scale=0.45]
\clip(-2.422798262991203,-6.267715333945327) rectangle (18.393532208113562,5.813685632891923);
\fill[line width=0.8pt,color=ccqqqq,fill=ccqqqq,fill opacity=0.03] (1.8214779402454433,-2.1802868460189853) -- (0.36301324969468696,2.3211473840759442) -- (7.5112908070854285,3.041376860891131) -- cycle;
\fill[line width=0.8pt,color=ccqqqq,fill=ccqqqq,fill opacity=0.03] (7.5112908070854285,3.041376860891131) -- (-0.3932277009612608,0.736642535082529) -- (1.8214779402454433,-2.1802868460189853) -- cycle;
\fill[line width=0.8pt,color=rvwvcq,fill=rvwvcq,fill opacity=0.04] (7.5112908070854285,3.0413768608911305) -- (12.756098492367299,-0.6836085781483103) -- (9.767146163584266,-3.1864060100810914) -- cycle;
\fill[line width=0.8pt,color=rvwvcq,fill=rvwvcq,fill opacity=0.04] (7.5112908070854285,3.0413768608911305) -- (13.201103673925418,-1.4420516322834165) -- (10.266168555903526,-4.0168720118977195) -- cycle;
\fill[line width=0.8pt,color=rvwvcq,fill=rvwvcq,fill opacity=0.04] (7.5112908070854285,3.0413768608911305) -- (13.738706660811122,-2.2354051085197035) -- (10.893800227391129,-4.648173855850587) -- cycle;
\fill[line width=0.8pt,color=qqwuqq,fill=qqwuqq,fill opacity=0.04] (3.247611861749814,-0.8714930572966189) -- (9.653973500610615,-2.0722524244967078) -- (12.232080316885316,-4.799885261592406) -- cycle;
\fill[line width=0.8pt,color=qqwuqq,fill=qqwuqq,fill opacity=0.04] (3.247611861749814,-0.8714930572966189) -- (10.57226608354998,-1.424045895363037) -- (13.192080316885315,-4.119885261592408) -- cycle;
\fill[line width=0.8pt,color=qqwuqq,fill=qqwuqq,fill opacity=0.04] (3.247611861749814,-0.8714930572966189) -- (11.490558666489346,-0.7938451031497469) -- (14.029367572262885,-3.296642535082529) -- cycle;
\fill[line width=0.8pt,color=ccqqqq,fill=ccqqqq,fill opacity=0.05] (1.8214779402454433,-2.1802868460189853) -- (2.1275754678918983,4.373801392999234) -- (7.5112908070854285,3.0413768608911305) -- cycle;
\draw [rotate around={31.607502246248927:(11.679618904153337,-2.648436005948856)},line width=0.8pt] (11.679618904153337,-2.648436005948856) ellipse (3.8251907238116476cm and 2.510733927770601cm);
\draw [line width=0.8pt] (14.029367572262885,-3.296642535082529)-- (11.490558666489346,-0.7938451031497469);
\draw [line width=0.8pt] (13.192080316885315,-4.119885261592408)-- (10.57226608354998,-1.424045895363037);
\draw [line width=0.8pt] (12.232080316885316,-4.799885261592406)-- (9.653973500610615,-2.0722524244967078);
\draw [line width=0.8pt] (13.738706660811122,-2.2354051085197035)-- (10.893800227391129,-4.648173855850587);
\draw [line width=0.8pt] (13.201103673925418,-1.4420516322834165)-- (10.266168555903526,-4.0168720118977195);
\draw [line width=0.8pt] (12.756098492367299,-0.6836085781483103)-- (9.767146163584266,-3.1864060100810914);
\draw [line width=0.8pt] (1.8214779402454433,-2.1802868460189853)-- (7.5112908070854285,3.041376860891131);
\draw [line width=0.8pt,color=ccqqqq] (1.8214779402454433,-2.1802868460189853)-- (0.36301324969468696,2.3211473840759442);
\draw [line width=0.8pt,color=ccqqqq] (0.36301324969468696,2.3211473840759442)-- (7.5112908070854285,3.041376860891131);
\draw [line width=0.8pt,color=ccqqqq] (7.5112908070854285,3.041376860891131)-- (-0.3932277009612608,0.736642535082529);
\draw [line width=0.8pt,color=ccqqqq] (-0.3932277009612608,0.736642535082529)-- (1.8214779402454433,-2.1802868460189853);
\draw [line width=0.8pt,color=rvwvcq] (7.5112908070854285,3.0413768608911305)-- (12.756098492367299,-0.6836085781483103);
\draw [line width=0.8pt,color=rvwvcq] (9.767146163584266,-3.1864060100810914)-- (7.5112908070854285,3.0413768608911305);
\draw [line width=0.8pt,color=rvwvcq] (7.5112908070854285,3.0413768608911305)-- (13.201103673925418,-1.4420516322834165);
\draw [line width=0.8pt,color=rvwvcq] (10.266168555903526,-4.0168720118977195)-- (7.5112908070854285,3.0413768608911305);
\draw [line width=0.8pt,color=rvwvcq] (7.5112908070854285,3.0413768608911305)-- (13.738706660811122,-2.2354051085197035);
\draw [line width=0.8pt,color=rvwvcq] (10.893800227391129,-4.648173855850587)-- (7.5112908070854285,3.0413768608911305);
\draw [line width=0.8pt,color=qqwuqq] (3.247611861749814,-0.8714930572966189)-- (9.653973500610615,-2.0722524244967078);
\draw [line width=0.8pt,color=qqwuqq] (12.232080316885316,-4.799885261592406)-- (3.247611861749814,-0.8714930572966189);
\draw [line width=0.8pt,color=qqwuqq] (3.247611861749814,-0.8714930572966189)-- (10.57226608354998,-1.424045895363037);
\draw [line width=0.8pt,color=qqwuqq] (13.192080316885315,-4.119885261592408)-- (3.247611861749814,-0.8714930572966189);
\draw [line width=0.8pt,color=qqwuqq] (3.247611861749814,-0.8714930572966189)-- (11.490558666489346,-0.7938451031497469);
\draw [line width=0.8pt,color=qqwuqq] (14.029367572262885,-3.296642535082529)-- (3.247611861749814,-0.8714930572966189);
\draw [line width=0.8pt,color=ccqqqq] (1.8214779402454433,-2.1802868460189853)-- (2.1275754678918983,4.373801392999234);
\draw [line width=0.8pt,color=ccqqqq] (2.1275754678918983,4.373801392999234)-- (7.5112908070854285,3.0413768608911305);
\begin{scriptsize}
\draw[color=black] (7.78072795489226,-3.55633798490686) node {\normalsize{$\sigma$}};
\draw[color=black] (13.359397566381254,-1.0957850359612173) node {\normalsize{$\mathcal{R}_2$}};
\draw[color=black] (13.50963409138269,-3.704716127661905) node {\normalsize{$\mathcal{R}_1$}};
\draw[color=black] (5.18352196210798,0.5026445765596264) node {\normalsize{$l$}};
\draw [fill=black] (3.247611861749814,-0.8714930572966189) circle (2.5pt);
\draw[color=black] (3.1992645120821055,-0.08528355678137356) node {\normalsize{$P_1$}};
\draw [fill=black] (7.5112908070854285,3.013768608911305) circle (2.5pt);
\draw[color=black] (7.71896203714104,3.4974035057653454) node {\normalsize{$P_2$}};
\end{scriptsize}
\end{tikzpicture}

    \caption{ Example($viii$): the red, green and blue planes correspond to the $k$-spaces of type $1$, $2$, $3$ in $\PG(n,q)/\gamma$, respectively.}

\end{figure}
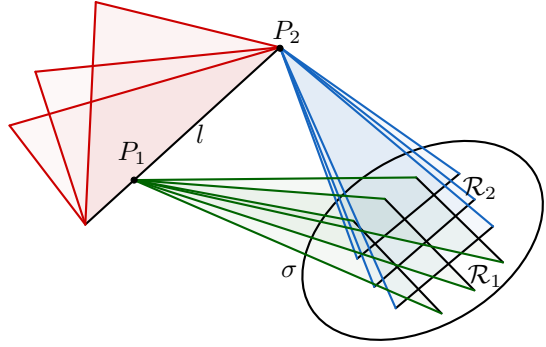

   \item[$(ix)$]  There is a $(k-3)$-space $\gamma$ contained in all $k$-spaces of $\mathcal{S}$. In the quotient space  $\PG(n,q)/ \gamma$, the set of planes corresponding to the elements of $\mathcal{S}$ is the set of planes of example $VII$ in \cite{EKRplanes}: 
    Let $\Psi$ be an $(n-k+2)$-space, disjoint to $\gamma$ in $\PG(n,q)$ and let $\rho$ be a $5$-space in $\Psi$. Consider a line $l$ and a $3$-space $\sigma$
disjoint to $l$. Choose three points $P_1$, $P_2$, $P_3$ on $l$ and choose four non-coplanar
points $Q_1$, $Q_2$, $Q_3$, $Q_4$ in $\sigma$. Denote $l_1= Q_1Q_2$, $\bar{l}_1=Q_3Q_4$, $l_2=Q_1Q_3$, $\bar{l}_2=Q_2Q_4$, $l_3=Q_1Q_4$, and $\bar{l}_3=Q_2Q_3$. Then $\mathcal{S}$ is the set containing all $k$-spaces through $\langle \gamma,l \rangle$ (type 0) and all $k$-spaces through $\langle \gamma, P_i \rangle $ in $\langle \gamma,l,l_i \rangle$ or in $\langle \gamma,l, \bar{l}_i \rangle$, $i=1,2,3$ (type $i$). Then $|\mathcal{S}|=\theta_{n-k}+ 6q^2$.\\

\textnormal{In Lemma \ref{maxmaarten7}, we prove that the set $\mathcal{S}$ is maximal.}

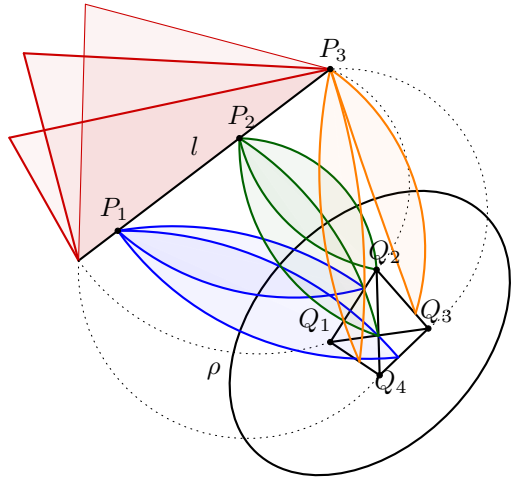
\begin{figure}[h!]
    \centering
\definecolor{ffxfqq}{rgb}{1,0.4980392156862745,0}
\definecolor{qqwuqq}{rgb}{0,0.39215686274509803,0}
\definecolor{qqqqff}{rgb}{0,0,1}
\definecolor{ccqqqq}{rgb}{0.8,0,0}
\begin{tikzpicture}[line cap=round,line join=round,>=triangle 45,x=1cm,y=1cm,scale=0.50]
\clip(-1.12,-7.7) rectangle (20.08,5.64);
\fill[line width=0.8pt,color=ccqqqq,fill=ccqqqq,fill opacity=0.04] (4.3,-1.66) -- (2.48,1.6) -- (10.92,3.42) -- cycle;
\fill[line width=0.8pt,color=ccqqqq,fill=ccqqqq,fill opacity=0.05] (4.3,-1.66) -- (2.86,3.84) -- (10.92,3.42) -- cycle;
\fill[line width=0.4pt,color=ccqqqq,fill=ccqqqq,fill opacity=0.05] (4.3,-1.66) -- (4.48,5.14) -- (10.92,3.42) -- cycle;
\draw [rotate around={46.78196059032664:(11.96,-3.58)},line width=0.8pt] (11.96,-3.58) ellipse (4.354353088322431cm and 2.970789595003832cm);
\draw [line width=0.8pt] (13.5,-3.46)-- (12.22,-4.7);
\draw [line width=0.8pt] (10.92,-3.82)-- (12.22,-4.7);
\draw [line width=0.8pt] (10.92,-3.82)-- (13.5,-3.46);
\draw [line width=0.8pt] (12.14,-1.9)-- (12.22,-4.7);
\draw [line width=0.8pt] (12.14,-1.9)-- (13.5,-3.46);
\draw [line width=0.8pt] (12.14,-1.9)-- (10.92,-3.82);
\draw [line width=0.8pt] (10.92,3.42)-- (4.3,-1.66);
\draw [line width=0.8pt,color=ccqqqq] (4.3,-1.66)-- (2.48,1.6);
\draw [line width=0.8pt,color=ccqqqq] (2.48,1.6)-- (10.92,3.42);
\draw [line width=0.8pt,color=ccqqqq] (4.3,-1.66)-- (2.86,3.84);
\draw [line width=0.8pt,color=ccqqqq] (2.86,3.84)-- (10.92,3.42);
\draw [line width=0.4pt,color=ccqqqq] (4.3,-1.66)-- (4.48,5.14);
\draw [line width=0.4pt,color=ccqqqq] (4.48,5.14)-- (10.92,3.42);
\draw [shift={(9.68045059857165,0.33585521245440064)},line width=0.4pt,dotted]  plot[domain=-0.7377930686569201:1.188644268276375,variable=\t]({1*3.323918138150528*cos(\t r)+0*3.323918138150528*sin(\t r)},{0*3.323918138150528*cos(\t r)+1*3.323918138150528*sin(\t r)});
\draw [shift={(11.242306382978724,-0.38288510638297807)},line width=0.4pt,dotted]  plot[domain=-0.9378060479488051:1.6553473884899492,variable=\t]({1*3.8165189029871582*cos(\t r)+0*3.8165189029871582*sin(\t r)},{0*3.8165189029871582*cos(\t r)+1*3.8165189029871582*sin(\t r)});
\draw [shift={(8.739742725880552,-1.9301439509954057)},line width=0.4pt,dotted]  plot[domain=3.0808208208239254:5.610963568242939,variable=\t]({1*4.447953824655634*cos(\t r)+0*4.447953824655634*sin(\t r)},{0*4.447953824655634*cos(\t r)+1*4.447953824655634*sin(\t r)});
\draw [shift={(9.039768177549306,1.6419746923038883)},line width=0.4pt,dotted]  plot[domain=3.7500689134373237:5.043923480198164,variable=\t]({1*5.7765421530119925*cos(\t r)+0*5.7765421530119925*sin(\t r)},{0*5.7765421530119925*cos(\t r)+1*5.7765421530119925*sin(\t r)});
\draw [shift={(9.978973525642951,4.338256844853014)},line width=0.8pt,color=qqqqff,fill=qqqqff,fill opacity=0.04]  plot[domain=3.983673875246295:4.980422970007183,variable=\t]({1*6.979453416579155*cos(\t r)+0*6.979453416579155*sin(\t r)},{0*6.979453416579155*cos(\t r)+1*6.979453416579155*sin(\t r)});
\draw [shift={(6.771538622028463,-9.339351168482855)},line width=0.8pt,color=qqqqff,fill=qqqqff,fill opacity=0.04]  plot[domain=0.9416973490179988:1.7392141890558928,variable=\t]({1*8.592292485546125*cos(\t r)+0*8.592292485546125*sin(\t r)},{0*8.592292485546125*cos(\t r)+1*8.592292485546125*sin(\t r)});
\draw [shift={(11.904973962833898,3.7929215292453122)},line width=0.8pt,color=qqqqff,fill=qqqqff,fill opacity=0.05]  plot[domain=3.7584141463454785:4.8124196388898834,variable=\t]({1*8.058758059412492*cos(\t r)+0*8.058758059412492*sin(\t r)},{0*8.058758059412492*cos(\t r)+1*8.058758059412492*sin(\t r)});
\draw [shift={(5.477774135225141,-10.333993582067006)},line width=0.8pt,color=qqqqff,fill=qqqqff,fill opacity=0.04]  plot[domain=0.7013759373289415:1.586272540726835,variable=\t]({1*9.466497950295246*cos(\t r)+0*9.466497950295246*sin(\t r)},{0*9.466497950295246*cos(\t r)+1*9.466497950295246*sin(\t r)});
\draw [shift={(3.945482853982031,-5.5104807104230575)},line width=0.8pt,color=qqwuqq,fill=qqwuqq,fill opacity=0.04]  plot[domain=0.22091403566867648:0.9971503564220695,variable=\t]({1*8.450129388826488*cos(\t r)+0*8.450129388826488*sin(\t r)},{0*8.450129388826488*cos(\t r)+1*8.450129388826488*sin(\t r)});
\draw [shift={(13.893323572911756,1.4279319865011544)},line width=0.8pt,color=qqwuqq,fill=qqwuqq,fill opacity=0.05]  plot[domain=3.111931909445377:4.389317789824955,variable=\t]({1*5.364331877868283*cos(\t r)+0*5.364331877868283*sin(\t r)},{0*5.364331877868283*cos(\t r)+1*5.364331877868283*sin(\t r)});
\draw [shift={(8.436152257405945,-2.122270905104563)},line width=0.8pt,color=qqwuqq,fill=qqwuqq,fill opacity=0.05]  plot[domain=0.05993892686372905:1.545136957605966,variable=\t]({1*3.7105110774090244*cos(\t r)+0*3.7105110774090244*sin(\t r)},{0*3.7105110774090244*cos(\t r)+1*3.7105110774090244*sin(\t r)});
\draw [shift={(13.052105182993095,2.6546902574728084)},line width=0.8pt,color=qqwuqq,fill=qqwuqq,fill opacity=0.04]  plot[domain=3.3735141515354634:4.514747040113817,variable=\t]({1*4.6451199345507295*cos(\t r)+0*4.6451199345507295*sin(\t r)},{0*4.6451199345507295*cos(\t r)+1*4.6451199345507295*sin(\t r)});
\draw [shift={(7.585103617309275,-1.3653923097972551)},line width=0.8pt,color=ffxfqq,fill=ffxfqq,fill opacity=0.04]  plot[domain=-0.29707012409213096:0.9621600518887308,variable=\t]({1*5.832796365548037*cos(\t r)+0*5.832796365548037*sin(\t r)},{0*5.832796365548037*cos(\t r)+1*5.832796365548037*sin(\t r)});
\draw [shift={(221.395736681109,72.47857889370302)},line width=0.8pt,color=ffxfqq]  plot[domain=3.4586322264941622:3.4896430084820236,variable=\t]({1*221.5155142428479*cos(\t r)+0*221.5155142428479*sin(\t r)},{0*221.5155142428479*cos(\t r)+1*221.5155142428479*sin(\t r)});
\draw [shift={(23.29605102888687,0.7309828202252172)},line width=0.8pt,color=ffxfqq,fill=ffxfqq,fill opacity=0.04]  plot[domain=2.9276422036464695:3.5537595855806248,variable=\t]({1*12.664811584178251*cos(\t r)+0*12.664811584178251*sin(\t r)},{0*12.664811584178251*cos(\t r)+1*12.664811584178251*sin(\t r)});
\draw [shift={(-6.061450494184168,-2.1881514582019856)},line width=0.8pt,color=ffxfqq,fill=ffxfqq,fill opacity=0.04]  plot[domain=-0.12075796681704354:0.31897444886455106,variable=\t]({1*17.883540579666004*cos(\t r)+0*17.883540579666004*sin(\t r)},{0*17.883540579666004*cos(\t r)+1*17.883540579666004*sin(\t r)});
\begin{scriptsize}
\draw[color=black] (7.86,-4.59) node {\normalsize{$\rho$}};
\draw [fill=black] (13.5,-3.46) circle (2pt);
\draw[color=black] (13.72,-3.01) node {\normalsize{$Q_3$}};
\draw [fill=black] (12.22,-4.7) circle (2pt);
\draw[color=black] (12.52,-4.9) node {\normalsize{$Q_4$}};
\draw [fill=black] (10.92,-3.82) circle (2pt);
\draw[color=black] (10.52,-3.25) node {\normalsize{$Q_1$}};

\draw [fill=black] (12.14,-1.9) circle (2pt);
\draw[color=black] (12.36,-1.45) node {\normalsize{$Q_2$}};

\draw[color=black] (7.36,1.39) node {\normalsize{$l$}};
\draw [fill=black] (5.331274436025437,-0.86862928474181) circle (2pt);
\draw[color=black] (5.24,-0.29) node {\normalsize{$P_1$}};
\draw [fill=black] (8.531351183671594,1.5870187330893806) circle (2pt);
\draw[color=black] (8.6,2.15) node {\normalsize{$P_2$}};
\draw [fill=black] (10.92,3.42) circle (2pt);
\draw[color=black] (11,3.93) node {\normalsize{$P_3$}};
\end{scriptsize}
\end{tikzpicture}
    \caption{Example$(ix)$: The red, blue, green and orange planes correspond to the $k$-spaces of type $0,1,2$ and $3$ respectively.}
\end{figure}
    
    \item[$(x)$]  $\mathcal{S}$ is the set of all $k$-spaces contained in a fixed $(k+2)$-space $\rho$. Then $|\mathcal{S}|=\qbin{k+3}{2}$.
\end{itemize}
\end{example}
From now on, let $\mathcal{S}$ be a maximal set of $k$-spaces pairwise intersecting in at least a $(k-2)$-space in the projective space $\PG(n,q)$ with $n \geq k+2$. 
We will study these relevant families focusing on the sets $\mathcal{S}$ such that  $|\mathcal{S}|> f(k,q)$.
In Section \ref{confi}, we investigate the sets $\mathcal{S}$ of $k$-spaces in $\PG(n,q)$  such that there is no point contained in all elements of $\mathcal{S}$ and such that $\mathcal{S}$ contains a set of three $k$-spaces that meet in a $(k-4)$-space.
In  Section \ref{noconfi}, we assume again that there is no point contained in all elements of $\mathcal{S}$ and that for any three $k$-spaces $X,Y,Z$ in $\mathcal{S}$, $\dim(X\cap Y\cap Z)\geq k-3$. In Section \ref{sectionsolidpointk}, we investigate the maximal sets $\mathcal{S}$ of $k$-spaces such that there is at least a point contained in all elements of $\mathcal{S}$. We end this article with the Main Theorem \ref{overzicht2} that classifies all sets of $k$-spaces pairwise intersecting in at least a $(k-2)$-space with size larger than $f(k,q)$.




\section{There are three elements of $\mathcal{S}$ that meet in a $(k-4)$-space}\label{confi}

 Suppose there exist three $k$-spaces $A,B,C$ in $\mathcal{S}$ with $\dim(A\cap B\cap C)=k-4$, and suppose that there is no point contained in all elements of $\mathcal{S}$. By the existence of Example \ref{voorbeeldenalgemeen}$(x)$, we may assume that the elements of $\mathcal{S}$ span at least a $(k+3)$-space. In this subsection, we will use the following notation.
 
\begin{notation}\label{notk}
Let $\mathcal{S}$ be a maximal set of $k$-spaces in $\PG(n,q)$ pairwise intersecting in at least a $(k-2)$-space. Let $A,B$ and $C$ in $\mathcal{S}$ be three $k$-spaces with $\pi_{ABC}=A\cap B\cap C$ a $(k-4)$-space. Let $\pi_{AB}=A\cap B$, $ \pi_{AC}= A \cap C$ and $\pi_{BC}=B\cap C$. Let $\mathcal{S}'$ be the set of $k$-spaces of $\mathcal{S}$ not contained in $\langle A,B \rangle$, and let $\alpha$ be the span of all subspaces $D':=D\cap \langle A,B \rangle$, $D\in \mathcal{S}'$.\end{notation}

Note explicitly, by Grassmann's dimension property, that $\pi_{AB}$, $\pi_{BC}$ and $\pi_{AC}$ are $(k-2)$-spaces and $\langle A,B\rangle =\langle B,C \rangle=\langle A,C \rangle $.\\ 


We first present a lemma that will be useful for the later classification results in this section.
\begin{lemma}\label{lemmahoespaceserbuiten}\emph{[Using Notation \ref{notk}]}
If there exist three $k$-spaces $A$, $B$ and $C$  in $\mathcal{S}$, with $\dim(A\cap B\cap C)=k-4$, then a $k$-space of $\mathcal{S}'$ meets $\langle A,B \rangle$ in a $(k-1)$-space. More specifically, it contains $\pi_{ABC}$ and meets $\pi_{AB}, \pi_{AC}$ and $\pi_{BC}$, each in a $(k-3)$-space through $\pi_{ABC}$. 
\end{lemma}


\begin{proof}
Consider a $k$-space $E$ of $\mathcal{S}'$. Clearly,
\begin{equation*}
    k-2 \leq \dim(E \cap \langle A,B \rangle) \leq k-1.
\end{equation*}
If  $\dim(E \cap \langle A,B \rangle)=k-2$, then this $(k-2)$-space has to lie in $A,B$ and $C$, and so in the $(k-4)$-space $\pi_{ABC}$, a contradiction. Hence, we know that $\dim(E \cap \langle A,B \rangle)=k-1$. By the symmetry of the $k$-spaces $A,B$ and $C$, it suffices to prove that $E$ contains $\pi_{ABC}$ and meets $\pi_{AB}$ in a $(k-3)$-space through  $\pi_{ABC}$. Using Grassmann's dimension property we find that 
\begin{equation*}
 \dim(E \cap \pi_{AB})\geq\dim(E\cap A)+\dim(E\cap B)-\dim(E\cap \langle A,B \rangle)\end{equation*}
is $k-2$ or $k-3$. If $\dim(E \cap \pi_{AB})=k-2$, then 
\begin{equation*}
   \dim(E \cap C)\leq \dim(E\cap \pi_{ABC})+\dim(E\cap \langle C, \pi_{AB} \rangle)-\dim(E\cap \pi_{AB})\leq (k-4)+(k-1)-(k-2)=k-3,
\end{equation*}
a contradiction since any two elements of $\mathcal{S}$ meet in at least a $(k-2)$-space. Hence, $\dim(E \cap \pi_{AB})=k-3$, and so \begin{equation*}
\dim(E\cap \pi_{ABC})\geq\dim(E\cap C)+\dim(E\cap \pi_{AB})-\dim(E\cap \langle C,\pi_{AB} \rangle) \geq (k-2)+(k-3)-(k-1)=k-4.
\end{equation*}This implies that the $(k-4)$-space $\pi_{ABC}$ is contained in $E$.
\end{proof}\\

Therefore, let $D$ be a $k$-space of $\mathcal{S'}$. By Lemma \ref{lemmahoespaceserbuiten}, for the remaining part of this section, we will denote by  $D'$  the $ (k-1) $-space $D \cap \langle A, B \rangle $.

\begin{gevolg}\emph{[Using Notation \ref{notk}]}\label{lemmagevolg}
Suppose $\mathcal{S}$ contains three elements $A,B$ and $C$, meeting in a $(k-4)$-space, and $\alpha$ is a $(k+i)$-space. {Up to a suitable labelling of $A,B$ and $C$, we have the following results. }
\begin{itemize}
    \item [$a)$] If $i=-1$, then $\alpha=D\cap \langle A,B \rangle$ for every $D\in \mathcal{S'}$.
    \item [$b)$] If $i=0$, then $\alpha=\langle \rho_1,\rho_2, \rho_3 \rangle$, with $\rho_1$ a $(k-3)$-space in $\pi_{AB}$, $\rho_2$ a $(k-3)$-space in $\pi_{BC}$, $\rho_3=\pi_{AC}$ and $\pi_{ABC}\subset \rho_j, j=1,2,3$. In this case, all elements of $\mathcal{S}'$ contain the $(k-2)$-space $\langle \rho_1,\rho_2 \rangle$. 
    \item [$c)$] If $i=1$, then $\alpha=\langle \rho_1,\rho_2, \rho_3 \rangle$, with $\rho_1$ a $(k-3)$-space in $\pi_{AB}$, $\rho_2=\pi_{BC}$, $\rho_3=\pi_{AC}$ and $\pi_{ABC}\subset \rho_j, j=1,2,3$. In this case, all elements of $\mathcal{S}'$ contain the $(k-3)$-space $ \rho_1$.
    \item [$d)$] If $i=2$, then $\alpha=\langle A,B \rangle$.
\end{itemize}
\end{gevolg}

\begin{proof}
For $i=-1$ and $i=2$, the corollary follows immediately from Lemma \ref{lemmahoespaceserbuiten}. Hence we start with the case that $\alpha$ is a $k$-space. Consider two elements of $\mathcal{S'}$, say $D_1$, $D_2$, meeting $\langle A, B \rangle$ in two different $(k-1)$-spaces $D'_1, D'_2$. These two elements of $\mathcal{S'}$ exist, as otherwise $\dim(\alpha)=k-1$. Since $D'_1$ and $D'_2$ span the $k$-space $\alpha$, they meet in a $(k-2)$-space. By Lemma \ref{lemmahoespaceserbuiten}, this $(k-2)$-space contains $\pi_{ABC}$, together with a $(k-3)$-space $\rho_1$ through $\pi_{ABC}$ in $\pi_{XY}$ and  a $(k-3)$-space $\rho_2$ through $\pi_{ABC}$ in $\pi_{YZ}$, with $\{X,Y,Z\}=\{A,B,C\}$. By Lemma \ref{lemmahoespaceserbuiten}, every other element of $\mathcal{S}'$ will meet $\langle A,B \rangle$ in a $(k-1)$-space through this $(k-2)$-space  $\pi_1 \cap \pi_2 = \langle \rho_1 , \rho_2 \rangle$, which proves the statement.

Suppose now that $\alpha$ is a $(k+1)$-space. Then, we consider two elements $D_3,D_4$ of $\mathcal{S}'$ meeting $\langle A, B \rangle$ in two $(k-1)$-spaces $D'_3, D'_4$ such that $\dim(D'_3\cap D'_4)=k-3$. These elements of $\mathcal{S}'$ exist as otherwise all elements of $\mathcal{S}'$ correspond to $(k-1)$-spaces pairwise intersecting in a $(k-2)$-space. But then, since these $(k-1)$-spaces span a $(k+1)$-space, they form a $(k-2)$-pencil (see Theorem \ref{basisEKRthm}). Using Lemma \ref{lemmahoespaceserbuiten}, and the proof above of the case $\dim(\alpha)=k$ or $i=0$, it follows that $\alpha$
 would be a $k$-space. Now, again by Lemma \ref{lemmahoespaceserbuiten}, we see that $D'_3\cap D'_4$ contains $\pi_{ABC}$ and a $(k-3)$-space $\rho_1$ through $\pi_{ABC}$ in $\pi_{XY}$, with $\{X,Y,Z\}=\{A,B,C\}$. Using dimension properties
 and the fact that $D'_3 \cap D'_4=\rho_1$, we see that every other element of $\mathcal{S}'$ will contain $\rho_1$, which proves the statement.
 \end{proof}\\

We distinguish between several cases depending on the dimension of $\alpha=\langle \,D\cap \langle A,B \rangle\,|D\in \mathcal{S}'\rangle$.

\subsection{$\alpha$ is a  $(k-1)$-space}\label{alphaplane}

\begin{prop}\label{prop5}\emph{[Using Notation \ref{notk}]} If $\mathcal{S}$ contains three $k$-spaces that meet in a $(k-4)$-space and $\dim(\alpha)=k-1$, then $\mathcal{S}$ is Example \ref{voorbeeldenalgemeen}$(v)$.
\end{prop}
\begin{proof}
From Corollary \ref{lemmagevolg}, we have that for all $ D \in \mathcal{S}', D \cap \langle A,B \rangle = \alpha$, so all the $k$-spaces in $\mathcal{S}'$ meet $\langle A,B \rangle$ in $\alpha$. 
As a $k$-space of $\mathcal{S}$ in $\langle A,B \rangle$ needs to have at least a $(k-2)$-space in common with every $D\in \mathcal{S'}$, we find that every $k$-space of $\mathcal{S}$ in $\langle A,B \rangle$ meets $\alpha$ in at least a $(k-2)$-space. Note that the condition that every two $k$-spaces of $\mathcal{S}$ in $\langle A,B \rangle$ meet in at least a $(k-2)$-space is fulfilled. Hence, $\mathcal{S}$ is Example \ref{voorbeeldenalgemeen}$(v)$ with $\rho = \langle A,B \rangle$.
\end{proof}


\subsection{$\alpha$ is a $k$-space}\label{alphasolid}
If $\alpha$ is a $k$-space, we can suppose w.l.o.g., by  Corollary \ref{lemmagevolg}, that $\alpha = \langle \pi_{AB},P_{AC},P_{BC}\rangle$ with $P_{AC}$ and $P_{BC}$  points in $\pi_{AC}\setminus \pi_{ABC}$ and $\pi_{BC}\setminus \pi_{ABC}$, respectively. 

\begin{prop}\label{prop4}\emph{[Using Notation \ref{notk}]} If $\mathcal{S}$ contains three $k$-spaces that meet in a $(k-4)$-space and $\dim(\alpha)=k$, then 
$\mathcal{S}$ is Example \ref{voorbeeldenalgemeen}$(iv)$.
\end{prop}
\begin{proof}
Recall that $\alpha = \langle \pi_{AB},P_{AC},P_{BC}\rangle$. By Corollary \ref{lemmagevolg}, we know that all the $k$-spaces $D\in \mathcal{S}'$ have a $(k-1)$-space $D'$ in common with $\alpha$ and they contain the $(k-2)$-space $\pi=\langle \pi_{ABC},P_{AC}P_{BC}\rangle$. So every pair of $k$-spaces in $\mathcal{S'}$ meets in a $(k-2)$-space inside $\langle A,B \rangle$. 
Consider a $k$-space $E$ of $\mathcal{S}$ in $\langle A,B \rangle$, not having a $(k-1)$-space in common with $\alpha$, and let $D_1$ and $D_2$ be $k$-spaces of $\mathcal{S}'$ with $D'_1 \cap D'_2=\pi$, and so $\langle D_1',D_2'\rangle = \alpha$. If $E $ does not contain $\pi$, then
\begin{equation*}
    \dim(E \cap \alpha) \geq \dim\langle E \cap D'_1 , E \cap D'_2\rangle\geq k-2+k-2-\dim(E \cap \pi) \geq k-1.
\end{equation*}
This is a contradiction. Hence, every $k$-space of $\mathcal{S}\setminus \mathcal{S}'$ contains $\pi$ or has a $(k-1)$-space in common with $\alpha$. From the maximality of $\mathcal{S}$, it follows that $\mathcal{S}$ is Example \ref{voorbeeldenalgemeen}$(iv)$ with $\rho=\langle A,B \rangle$ and $\pi=\langle \pi_{ABC},P_{AC}P_{BC}\rangle$.
\end{proof}


\subsection{$\alpha$ is a $(k+1)$-space}
To understand the structure of these sets of $k$-spaces, we will first investigate the case $ k = 3 $ and then we will generalize our results to $k\geq 3$.

\subsubsection{$k=3$ and $\alpha$ is a $4$-space}\label{alpha_4}

Note that for $k=3$, the spaces $\pi_{AB}, \pi_{BC}$ and $\pi_{AC}$ are pairwise disjoint lines and $\pi_{ABC}$ is the empty space. By  Corollary \ref{lemmagevolg}, we can suppose w.l.o.g. that $\alpha = \langle P_{AB},\pi_{AC},\pi_{BC}\rangle$, with $P_{AB}$ a point in $\pi_{AB}\setminus \pi_{ABC}$. Hence, all planes $D'=D\cap \langle A,B \rangle$, $D\in \mathcal{S}'$, contain $P_{AB}$ and span the $4$-space $\alpha$.

From now on, let $\mathcal{L}$ be the set of lines $D\cap C$, $D \in \mathcal{S}'$.

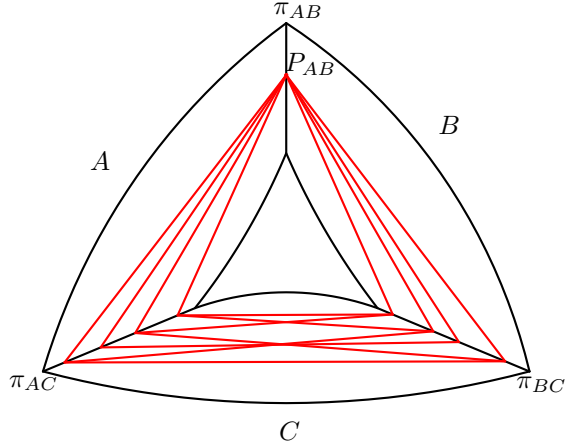
\begin{figure}[h] 
\centering
\definecolor{ffqqqq}{rgb}{1,0,0}
\begin{tikzpicture}[line cap=round,line join=round,>=triangle 45,x=1cm,y=1.05cm, scale=0.4]
\clip(-14.8148293340618,-4.4090286207433795) rectangle (12.907239857010698,10);
\draw [line width=0.8pt] (-3,0)-- (-8,-2);
\draw [line width=0.8pt] (3,0)-- (8,-2);
\draw [shift={(0,-8.75)},line width=0.8pt]  plot[domain=1.240498971965643:1.9010936816241504,variable=\t]({1*9.25*cos(\t r)+0*9.25*sin(\t r)},{0*9.25*cos(\t r)+1*9.25*sin(\t r)});
\draw [line width=0.8pt] (0,9)-- (0,4.89);
\draw [shift={(12.184695481335961,-8.270687622789788)},line width=0.8pt]  plot[domain=2.185212748834585:2.8403799855141356,variable=\t]({1*21.136306558546917*cos(\t r)+0*21.136306558546917*sin(\t r)},{0*21.136306558546917*cos(\t r)+1*21.136306558546917*sin(\t r)});
\draw [shift={(-19.076615999999994,13.2282)},line width=0.8pt]  plot[domain=5.594682347837613:5.871118964360154,variable=\t]({1*20.819290507878883*cos(\t r)+0*20.819290507878883*sin(\t r)},{0*20.819290507878883*cos(\t r)+1*20.819290507878883*sin(\t r)});
\draw [shift={(20.656043478260884,14.197173913043487)},line width=0.8pt]  plot[domain=3.564927750876754:3.818826858464239,variable=\t]({1*22.656028302053326*cos(\t r)+0*22.656028302053326*sin(\t r)},{0*22.656028302053326*cos(\t r)+1*22.656028302053326*sin(\t r)});
\draw [shift={(-9.667822445561141,-6.440234505862649)},line width=0.8pt]  plot[domain=0.2462183601117205:1.0113742127191454,variable=\t]({1*18.217234489211787*cos(\t r)+0*18.217234489211787*sin(\t r)},{0*18.217234489211787*cos(\t r)+1*18.217234489211787*sin(\t r)});
\draw [shift={(0,29.5)},line width=0.8pt]  plot[domain=4.463678991291166:4.961098969478212,variable=\t]({1*32.5*cos(\t r)+0*32.5*sin(\t r)},{0*32.5*cos(\t r)+1*32.5*sin(\t r)});
\draw [line width=0.8pt,color=ffqqqq] (0,7.358138923355905)-- (7.198373347153097,-1.6793493388612386);
\draw [line width=0.8pt,color=ffqqqq] (0,7.358138923355905)-- (-3.572078647048569,-0.2288314588194275);
\draw [line width=0.8pt,color=ffqqqq] (0,7.358138923355905)-- (-4.964881291150941,-0.7859525164603763);
\draw [line width=0.8pt,color=ffqqqq] (0,7.358138923355905)-- (-7.290586243828808,-1.7162344975315231);
\draw [line width=0.8pt,color=ffqqqq] (-7.290586243828808,-1.7162344975315231)-- (7.198373347153097,-1.6793493388612386);
\draw [line width=0.8pt,color=ffqqqq] (-4.964881291150941,-0.7859525164603763)-- (7.198373347153097,-1.6793493388612386);
\draw [line width=0.8pt,color=ffqqqq] (0,7.358138923355905)-- (4.834860911551517,-0.7339443646206067);
\draw [line width=0.8pt,color=ffqqqq] (4.834860911551517,-0.7339443646206067)-- (-7.290586243828808,-1.7162344975315231);
\draw [line width=0.8pt,color=ffqqqq] (0,7.358138923355905)-- (3.501864286544214,-0.2007457146176856);
\draw [line width=0.8pt,color=ffqqqq] (-3.572078647048569,-0.2288314588194275)-- (3.501864286544214,-0.2007457146176856);
\draw [line width=0.8pt,color=ffqqqq] (-6.102663732812033,-1.2410654931248133)-- (0,7.358138923355905);
\draw [line width=0.8pt,color=ffqqqq] (0,7.358138923355905)-- (5.679344476619753,-1.071737790647901);
\draw [line width=0.8pt,color=ffqqqq] (-4.964881291150941,-0.7859525164603763)-- (3.501864286544214,-0.2007457146176856);
\draw [line width=0.8pt,color=ffqqqq] (-3.572078647048569,-0.2288314588194275)-- (4.834860911551517,-0.7339443646206067);
\draw [line width=0.8pt,color=ffqqqq] (-6.102663732812033,-1.2410654931248133)-- (5.679344476619753,-1.071737790647901);
\begin{scriptsize}
	\draw[color=black] (-8.3,-2.3579324510360013) node {\normalsize{$\pi_{AC}$}};
	\draw[color=black] (8.4,-2.4) node {\normalsize{$\pi_{BC}$}};
	\draw[color=black] (0.40,9.4) node {\normalsize{$\pi_{AB}$}};
\draw[color=black] (-6.1164825675627545,4.651638940254569) node {\normalsize{$A$}};
\draw[color=black] (5.366586724901813,5.778043557499055) node {\normalsize{$B$}};
\draw[color=black] (0.1100318444275692,-3.890262740516124) node {\normalsize{$C$}};
\draw [color=ffqqqq, fill=ffqqqq] (0,7.358138923355905) circle (1.5pt);
\draw[color=black] (0.8,7.7) node {\normalsize{$P_{AB}$}};
\end{scriptsize}
\end{tikzpicture}

\caption{There are three solids $A,B,C$ in $\mathcal{S}$, with $A\cap B\cap C=\emptyset$ and $\dim(\alpha) = 4$} \label{alpha4figure}\end{figure}

\begin{prop}\label{lemma4}\emph{[Using Notation \ref{notk}]}
If $\mathcal{S}$ contains three solids such that there is no point contained in the three of them,  and if $\dim(\alpha)=4$, then a solid of $\mathcal{S}$ in $\langle A,B \rangle$ either
\begin{itemize}
\item [i)] is contained in $\alpha$, or
\item [ii)]contains $P_{AB}$ and a line $r$ of $C$, intersecting all lines of $\mathcal{L}$.
\end{itemize}
\end{prop}
\begin{proof} Recall that all intersection planes $D\cap \langle A,B \rangle$ contain $P_{AB}$ and span the $(k+1)$-space $\alpha$. Hence, we can see that there exist solids $D_1,D_2 \in \mathcal{S}'$, such that their intersection planes $D'_1$ and $D'_2 $ with $\alpha$, meet exactly in the point $P_{AB}$. Indeed, by Theorem \ref{basisEKRthm}, if all the planes $D\cap \langle A,B \rangle $, $D\in \mathcal{S}'$, would pairwise intersect in a line, then these planes lie in a fixed solid or contain a fixed line. Neither possibility can occur since $\alpha$ is a $4$-space, and $P_{AB}$ is the only point contained in all intersection planes. 

 Suppose first that $E$ is a solid of $\mathcal{S}$ in $\langle A,B \rangle$, not containing $P_{AB}$. As $E$ needs to contain at least a line of every plane $D'=D\cap \langle A,B \rangle$, $D\in \mathcal{S}'$, $E$ contains at least
a line $l_1\subset D_1'\subset \alpha$ and a line $l_2\subset D_2'\subset \alpha$. Note that $l_1$ and $l_2$ are disjoint as they do not contain the point $P_{AB}$. Hence, $E=\langle l_1,l_2 \rangle \subset \alpha$. 


 So now we can suppose that $E$ contains the point $P_{AB}$
 and meets $\alpha$ in precisely the plane $\gamma$. The plane $\gamma$ is the span of $P_{AB}$ and the line $r=\gamma \cap C$. As $E\cap D$ is at least a line  of the plane $D'=D\cap \langle A,B \rangle$ for every $D\in \mathcal{S}'$, and since every two lines in the plane $\gamma$ meet each other, we have that $r$ has to intersect all the lines of $\mathcal{L}$. Hence, we find the second possibility.
\end{proof}\\

In the previous proposition, we proved that there are two types of solids of $\mathcal{S}$ contained in $\langle A,B \rangle$. One of them are the solids containing $P_{AB}$ and a line $r\subset C$, intersecting all lines of $\mathcal{L}$. The number of these solids depends on the number of lines $r$ meeting all lines of $\mathcal{L}$. \\

\noindent
\textbf{Case 1. There is a line $l\in \mathcal{L}$ that intersects all the lines of $\mathcal{L}$} 
\vspace{0.3cm}

Note that there cannot be two lines in $\mathcal{L}$ intersecting all the lines of $\mathcal{L}$, since then all lines of $\mathcal{L}$ would lie in a plane or go through a fixed point in $C$. This gives a contradiction as the lines of $\mathcal{L}$ span $C$ and at least two points of both $\pi_{AB}$ and $\pi_{BC}$ are covered by the lines of $\mathcal{L}$.\\

\begin{prop}\label{prop6k=3}
$\mathcal{S}$ is  Example \ref{voorbeeldenalgemeen}(vi) for $k=3$.
\end{prop}
\begin{proof} Let $P_A=l \cap \pi_{AC}$, $P_{B}=l \cap \pi_{BC}$, $\pi_A=\langle \pi_{AC},l \rangle$ and $\pi_B=\langle \pi_{BC},l \rangle$. Since every line $m \neq l$ of $ \mathcal{L}$ intersects the lines $\pi_{AC}, \pi_{BC}$ and $l$, it follows that $m$ contains the point $P_A$ and is contained in $\pi_B$, or $m$ contains the point $P_B$ and is contained in $\pi_A$. Note that since $\dim(\alpha)=4$, there is at least one line $m_1\neq l$ in $\mathcal{L}$ through $P_A$ and there is at least one line $m_2\neq l$ in $\mathcal{L}$ through $P_B$.   As a consequence of  Proposition $\ref{lemma4}$, we have that a solid of $\mathcal{S}$ in $\langle A,B \rangle$, not contained in $\alpha$, contains $P_{AB}$ and it meets $C$ in a line $r$ that meets all lines of $\mathcal{L}$. Hence, $r$ is a line of the plane $\pi_A$ through $P_A$ or  in a line of $\pi_B$ through $P_B$. \\
Consider now the set $\mathcal{F}$ of solids of $\mathcal{S}'$, not through $\langle P_{AB},l\rangle$. We will prove that these solids lie in a $5$-space  that meets $\langle A,B \rangle$ in $\alpha$. Let $E_A,E_B\in \mathcal{F}$ be two solids through $m_1 \ni P_A$ and $m_2\ni P_B$ respectively. Since the planes $E_A\cap \alpha$ and $E_B\cap \alpha$ meet in precisely the point $P_{AB}$, the solids $E_A$ and $E_B$ have precisely a line in common, and so, they span a $5$-space $\rho_2$ through $\alpha$. Then every other solid $F\in \mathcal{F}$ is contained in $\rho_2$ as it meets $E_A\cap \alpha$, or $E_B\cap \alpha$, precisely in one point, namely $P_{AB}$, and so it must contain at least a point of $E_A$, or $E_B$ respectively, in $\rho_2\setminus \alpha$. This point, together with the plane $F\cap \alpha$, spans $F$ and so $F\subset \rho_2$.
Hence, $\mathcal{S}$ is Example \ref{voorbeeldenalgemeen}$(vi)$, with $\rho_1=\langle A,B \rangle$, $\lambda_A=P_A, \lambda_B=P_B$ and $\lambda=l$.
\end{proof}\\

\noindent \textbf{Case 2. For every line in $\mathcal{L}$, there exists another line in $\mathcal{L}$ disjoint to the given line} 
\vspace{0.3cm}

Depending on the structure of the set of lines $\mathcal{L}$,  we discuss the set of the solids of $\mathcal{S}$ in $\langle A,B \rangle$ not contained in $\alpha$. We have different possibilities for the line $r$ of $C$, meeting all lines of $\mathcal{L}$ (see Proposition $\ref{lemma4}$):

\begin{itemize}
    \item [$i)$] \textnormal{Suppose there are three pairwise disjoint lines in $\mathcal{L}$, then these three lines belong to a unique regulus $\mathcal{R}$.}
    \begin{itemize}
        \item [$a)$] \textnormal{If $\mathcal{L}$ is contained in $\mathcal{R}$, then $|\mathcal{L}|\leq q+1$ and $r$ is a line of the opposite regulus $\mathcal{R}^c$. Hence, there are $q+1$ possibilities for $r$.}
        \item [$b)$] \textnormal{If $\mathcal{L}$ is not
contained in $\mathcal{R}$, then there are exactly two lines, intersecting all the lines of $\mathcal{L}$, namely $\pi_{AC}$ and $\pi_{BC}$. Let $l\in\mathcal{L}\setminus \mathcal{R}$. If there were a third line $r$ meeting all lines of $\mathcal{L}$, then $r\in \mathcal{R}^c$. But then there would be three lines, namely $r$, $\pi_{AC}$ and $\pi_{BC}$,  in $\mathcal{R}^c$, all of them intersecting $l$. Hence, $l$ also has to lie in $\mathcal{R}$, a contradiction.
        In this case there are at most $2$ possibilities for $r$ and $|\mathcal{L}|\leq (q+1)^2$. \\
Note that in this case, $\mathcal{L}$ is not contained in any regulus. This follows since the three pairwise disjoint lines in $\mathcal{L}$ define a unique regulus. }
    \end{itemize}
    \item [$ii)$]\textnormal{Suppose there are no three pairwise disjoint lines in $\mathcal{L}$.}
    In this case, we can prove the following lemma.

\begin{lemma}\label{lemmaopsplitsing}
The set $\mathcal{L}$ is contained in the union of two point-pencils such that their vertices are contained either in $\pi_{AC}$ or in $\pi_{BC}$. 
\end{lemma}
\proof
\textcolor{black}{We can suppose that $\mathcal{L}$ contains at least two disjoint lines $l_1$, $l_2$, since the lines of $\mathcal{L}$ span
the solid $C$. Let $P_i=\pi_{AC} \cap l_i$ and $Q_i=\pi_{BC}\cap l_i$, for $i=1,2$. As there are no three pairwise disjoint lines in $\mathcal{L}$ we see that every line $l\in \mathcal{L}$ contains at least one of the points $P_i$ and $Q_i$, with $i=1,2$, and so $\mathcal{L}$ is contained in the union of $4$ point-pencils with vertices $P_1,P_2,Q_1,Q_2$. If $|\mathcal{L}|\leq 4$, then it is easy to see that $\mathcal{L}$ is contained in the union of two point-pencils.
 Suppose now that $|\mathcal{L}|\geq 5$ and
 that $\mathcal{L}$ is not contained in the union of two of these point-pencils. 
 Due to the symmetry, we can suppose w.l.o.g. that $\mathcal{L}$ contains a line $l_3\neq l_1, P_1 Q_2$ that contains $P_1$, a line $l_4\neq l_2, P_1 Q_2$ that contains $Q_2$ and a line $l_5\neq l_2, P_2 Q_1$ that contains $P_2$. Let $Q_3=l_3\cap \pi_{BC}$ and $P_4=l_4\cap \pi_{AC}$. Then $l_5$ contains the point $Q_3$ as otherwise $l_3,l_4$ and $l_5$ would be pairwise disjoint. So $l_5=P_2 Q_3$, but then we see that $l_1,l_4$ and $l_5$ are three pairwise disjoint lines, a contradiction. Hence, $\mathcal{L}$ is contained in the union of two point-pencils. }\endproof
    
By using the notations in the lemma above, since $\mathcal{L}$ contains no $3$ pairwise disjoint lines, we know that if $|\mathcal{L}|\geq 3$, then every line $l_0 \in \mathcal{L} \setminus \{l_1,l_2\}$ contains at least one of the points $P_1,P_2,Q_1,Q_2$. From Lemma \ref{lemmaopsplitsing}, we find the following possibilities for the set $\mathcal{L}$.
    \begin{itemize}
        \item [$a)$] \textnormal{If $\mathcal{L}$ only contains two lines $l_1,l_2$, then $l_1$ and $l_2$ are disjoint and we find $(q+1)^2$ possibilities for $r$, as every such line is defined by a point of $l_1$ and a point of $l_2$.}
        \item [$b)$] \textnormal{If $\mathcal{L}$ contains at least $3$ elements and is contained in the union of a line $l_0$ and a point-pencil through a point $P$, then $|\mathcal{L}|\leq q+2$. Let $P_0=l_0\cap \pi_{AC}$, $Q_0=l_0\cap \pi_{BC}$ and suppose w.l.o.g. that $P \in \pi_{AC}$.  A line $r$ that meets all lines of $\mathcal{L}$ is a line that contains $P$ and a point of $l_0$ or is a line that contains $Q_0$ and lies in the plane $\langle P, \pi_{BC} \rangle$. Hence, there are at most $2q+1$ possibilities for the line $r$.}
        \item [$c)$] \textnormal{If $\mathcal{L}$ contains at least $4$ elements and is contained in the union of two point-pencils through the points $P$ and $Q$ respectively such that $\mathcal{L}$ contains at least two lines through $P$ and at least two lines through $Q$, then $|\mathcal{L}|\leq 2q$. {Note that $|\mathcal{L}|< 2q+1$, as the line $PQ$ is not contained in $\mathcal{L}$. This follows  since this line meets all other lines of $\mathcal{L}$, and so, this situation is discussed in Case 1}. Therefore, a line $r$ that meets all lines of $\mathcal{L}$ is the line $\pi_{AC}$, the line $\pi_{BC}$ or the line $PQ$ if $PQ \neq \pi_{AC}$, $\pi_{BC}$. Hence, there are at most $3$ possibilities for the line $r$.}

    \end{itemize}
\end{itemize}

For every intersection
plane $D'$ in $\alpha$, there are at most $\qbin{3}{1}-\qbin{2}{1}=q^2$ ways to extend the plane to a solid $D\in \mathcal{S}'$, as this solid also has to meet several solids of $\mathcal{S}'$ in a point $Q \notin \langle A,B \rangle$. And since the number of planes $D'$ equals the number of lines in $\mathcal{L}$, there are at most $(q+1)\cdot q^2,(q+1)^2\cdot q^2, 2 \cdot q^2,  (q+2)\cdot q^2,2(q+1)\cdot q^2$ solids outside of $\langle A,B \rangle$, respectively, dependent on the five cases $ia), ib), iia), iib), iic)$ above. \\
For the solids inside $\langle A,B \rangle$, there are $\qbin{5}{1}$ solids in $\alpha$ and $(q+1) \cdot q^2,2 \cdot q^2,(q+1)^2 \cdot q^2, (2q+1) \cdot q^2, 3 \cdot q^2$ solids of the second type of Proposition \ref{lemma4}, respectively. We find these numbers by multiplying the number of possibilities for
the line $r$ and the number $q^2$ of $3$-spaces through a plane in $\langle A,B \rangle$, not contained in $\alpha$. So, in total, we have at most $\qbin{5}{1}+(q^2+2q+3 ) \cdot q^2 =O(2q^4)$ solids, using case $ib)$ or $iia)$.

\begin{remark}
Note that the number of elements of $\mathcal{S}$ in this case is smaller than $f(3,q)=3q^4+6q^3+5q^2+q+1$, and so we will not consider these maximal sets of solids in our classification result (Main Theorem \ref{overzicht2}).
\end{remark}

\subsubsection{General case $k>3$ and $\alpha$ is a $(k+1)$-space}
By  Corollary \ref{lemmagevolg}, we can suppose w.l.o.g., that $\alpha$ is spanned by $\pi_{AC}, \pi_{BC}$ and a point $P_{AB}$
of $\pi_{AB}$ outside of $\pi_{ABC}$, and that all $(k-1)$-spaces $D'=D\cap \langle A,B \rangle, D\in \mathcal{S}'$, contain $\langle P_{AB}, \pi_{ABC} \rangle$.
\begin{prop}\label{prop6kmeerdan3}
\emph{[Using Notation \ref{notk}]} If $\mathcal{S}$ contains three $k$-spaces that meet in a $(k-4)$-space and $\dim(\alpha)=k+1$, then a $k$-space of $\mathcal{S}$ in $\langle A,B \rangle$ is contained in $\alpha$ or  contains $\pi_{ABC}$. More specifically, if $|\mathcal{S}|> f(k,q)$, then $\mathcal{S}$ is Example \ref{voorbeeldenalgemeen}$(vi)$.
\end{prop}
\begin{proof}
We suppose that $E$ is a $k$-space of $\mathcal{S}$ in $\langle A,B \rangle$, not through $\pi_{ABC}$. As $E$ contains at least a $(k-2)$-space of all the $(k-1)$-spaces $D'$, with $D\in \mathcal{S}'$, we find that $E$ contains a hyperplane $\tau_0$ of $\pi_{ABC}$, a $(k-4)$-space $\tau_1$ of $\alpha \cap \pi_{AB}$,  a $(k-3)$-space $\tau_2$ of  $ \pi_{AC}$  and a $(k-3)$-space $\tau_3$ of $ \pi_{BC}$. As $\tau_1\cap \tau_2=\tau_1\cap \tau_3=\tau_2\cap \tau_3=\tau_0$, and by the Grassmann
dimension property, we see that $E \subset \alpha$.

For the $k$-spaces through $\pi_{ABC}$, we can investigate the solids $E/ \pi_{ABC}$, $ E\in \mathcal{S}$, 
in the quotient space $\PG(n,q) / \pi_{ABC}$, and use the results in Case $1$ and Case $2$ of Subsection \ref{alpha_4}. These results imply that a $k$-space in $\langle A,B \rangle$ through $\pi_{ABC}$ is contained in $\alpha$ or contains $\langle P_{AB}, \pi_{ABC} \rangle$ and a line in $C\setminus \pi_{ABC}$ that meets all the $(k-2)$-spaces $D\cap C, D\in \mathcal{S}'$. Then there are two cases:
\begin{itemize}
    \item [-] \textsc{Case 1.} If there is a line $l\in C/ \pi_{ABC}$ meeting the subspaces $D\cap C$ for all $D\in \mathcal{S}'$, then there are $\theta_{n-k}+q^4+5q^3+q^2$ $k$-spaces of $\mathcal{S}$ that contain $\pi_{ABC}$.
    \item [-] \textsc{Case 2.} If there is no line $l\in C/ \pi_{ABC}$ meeting the subspaces $D\cap C$ for all $D\in \mathcal{S}'$, then there are at most $2q^4+3q^3+4q^2+q+1$ $k$-spaces of $\mathcal{S}$ that contain $\pi_{ABC}$.
\end{itemize}
{It is clear that two elements of $\mathcal{S}$ in $\alpha$ meet in at least a $(k-1)$-space. From the investigation of the quotient space $\PG(n,q)/\pi_{ABC}$ there follows that two elements of $\mathcal{S}$ through $\pi_{ABC}$, not in $\alpha$, meet in at least a $(k-2)$-space.} A $k$-space $E_1$ of $\mathcal{S}$ in $\alpha$ and a $k$-space $E_2$ of $\mathcal{S}$ not in $\alpha$, but through $\pi_{ABC}$, will also meet in a $(k-2)$-space. This follows since $E_2$ contains the $(k-3)$-space $\langle P_{AB}, \pi_{ABC} \rangle \subset \alpha$ and a line in $C\setminus \pi_{ABC} \subset \alpha$. Hence, $E_2$ meets $\alpha$ in a $(k-1)$-space. Since $E_1$ is contained in $\alpha$, it follows that $E_1$ and $E_2$ meet in at least a $(k-2)$-space.

 Now, as every element of $\mathcal{S}$, not through $\pi_{ABC}$, is contained in $\alpha$, there are $\theta_{k+1}-\theta_4$ elements of $\mathcal{S}$ not through $\pi_{ABC}$. Hence, in \textsc{Case 1}, $\mathcal{S}$ is Example \ref{voorbeeldenalgemeen}$(vi)$ and $|\mathcal{S}|=\theta_{n-k}+\theta_{k+1}+4q^3-q-1$. In \textsc{Case 2}, $|\mathcal{S}|\leq \theta_{k+1}+q^4+2q^3+3q^2$, which proves the proposition.
 \end{proof}\\


\subsection{$\alpha$ is a $(k+2)$-space}\label{sectionbla}
Here again, we first consider the case $k=3$. 
\subsubsection{$\alpha$ is a 5-space}\label{sectie5} 
We start with a lemma that will often be used in this subsection.

\begin{lemma}\label{hulplemmatje}\emph{[Using Notation \ref{notk}]}
If $\mathcal{S}$ contains three solids $A,B,C$, with $A\cap B\cap C=\emptyset$, then every two intersection planes $D'_1$ and $D'_2$, with $D'_i=D_i\cap \langle A,B \rangle, D_i \in \mathcal{S}', i=1,2$, share a point on $\pi_{AB}$, $\pi_{AC}$ or $\pi_{BC}$.
\end{lemma}
\begin{proof}
Consider two solids $D_1$ and $D_2$ in $\mathcal{S}'$, with corresponding intersection planes $D'_1$ and $D'_2$ in $\langle A,B \rangle$. Since $D_1$ and $D_2$ meet in at least a line, $D'_1$ and $D'_2$ have to meet in at least a point. If $D'_1$ and $D'_2$ do not meet in a point of $\pi_{AB}$, $\pi_{AC}$ or $\pi_{BC}$, then these planes define $6$ different intersection points $P_1,\dots, P_6$ on the lines $\pi_{AB}$, $\pi_{AC}$ and $\pi_{BC}$. As $\langle D'_1,D'_2\rangle =\langle P_1,\dots, P_6\rangle=\langle \pi_{AB},\pi_{AC},\pi_{BC}\rangle$, we find that $D'_1$ and $D'_2$ span a $5$-space, so these planes are disjoint, a contradiction.\end{proof}\\

If $\alpha$ is a $5$-space,  we distinguish two cases, depending on the planes $D'=D\cap \langle A,B \rangle$, $D \in \mathcal{S}'$. 
\begin{lemma} \label{alpha5space}\emph{[Using Notation \ref{notk}]}
If $\mathcal{S}$ contains three solids $A,B,C$, with $A\cap B\cap C=\emptyset$, and if $\dim(\alpha)=5$, then we have one of the following possibilities for the planes $D'=D\cap \langle A,B \rangle, D\in \mathcal{S}'$:
\begin{itemize}
    \item [i)] There are four possibilities for the planes $D'$: $\langle P_1,P_3,P_6\rangle,\langle P_1,P_4,P_5\rangle,\langle P_2,P_4,P_6\rangle$ and 
    \\$\langle P_2,P_3,P_5\rangle$, where $P_1,P_2 \in \pi_{AB}, P_3,P_4 \in \pi_{BC}$ and $P_5,P_6 \in \pi_{AC}$.
    \item  [ii)]There are three points $P\in \pi_{AB},Q\in \pi_{BC}$ and $R\in \pi_{AC}$ so that every plane $D'$ contains at least two of the three points of $\{P,Q,R\}$.
\end{itemize}
\end{lemma}

\begin{proof}
We prove the Lemma by construction and we start with a plane, we say $D_1'$, intersecting $\pi_{AB},\pi_{BC}$ and $\pi_{AC}$ in the points  $P,Q$ and $R'$ respectively.\\
\noindent
\emph{Case $(a)$: There exists a plane $D_2'$ such that $D_1'\cap D_2'$ is a point  \textnormal{(w.l.o.g. $P$,  see Lemma \ref{hulplemmatje})} and let $D_2'\cap \pi_{BC}$ be $Q'$ and $D_2'\cap \pi_{AC}$ be $R$.} In this case we know that there exists a third plane $D_3'$ intersecting $\pi_{AB}$ in a point $P'$ different from $P$ (as $\dim(\alpha)=5$). 
Then $D_3'$ needs at least a point of $D_2'$ and $D_1'$. This implies that $D_3'$ contains $Q$ and $R$ or $Q'$ and $R'$ (w.l.o.g. $Q$ and $R$) by Lemma \ref{hulplemmatje}. Now there are two possibilities:
\begin{itemize}
    \item [$i)$] There exists a plane $D_4'=\langle P',Q',R'\rangle$, and then, by construction, we cannot add another plane $D_i'$. (In the formulation of the lemma $P=P_1,P'=P_2,Q=P_3,Q'=P_4,R=P_5,R'=P_6$.)
    \item [$ii)$] There exists no plane $D_4'=\langle P',Q',R'\rangle$, then, by construction, we see that all the planes need to contain at least two of the three points $P,Q,R$ by Lemma \ref{hulplemmatje}.
\end{itemize}
\emph{Case $(b)$: all the planes $D_i'$ intersect pairwise in a line.} Then all these planes have to lie in a solid (contradiction since they span a $5$-space) or they go through a fixed line $l$. In this last case, $l$ cannot be one of the lines $\pi_{AB},\pi_{AC},\pi_{BC}$ and also, $l$ cannot intersect one of these lines, as otherwise all the planes $D_i'$ would contain the intersection point of this line and $l$ (which gives a contradiction since $\dim(\alpha)=5$). Consider now the disjoint
lines $l$ and $\pi_{AB}$. Then all the planes $D_i'$ would contain $l$ and a point of $\pi_{AB}$, but this implies that $\dim(\alpha)=3$ which also gives a contradiction. We  conclude that this case cannot happen.
\end{proof}

\noindent
\subsubsection*{\textbf{Case 1. There are four intersection planes $D'$}} \label{sectie4}

In this situation, using the notation from Lemma \ref{alpha5space}, there are four possibilities for the planes $D'=D\cap \langle A,B \rangle$, $D\in \mathcal{S}'$: 
$\langle P_1,P_3,P_6\rangle,\langle P_1,P_4,P_5\rangle,\langle P_2,P_4,P_6\rangle$ and 
    $\langle P_2,P_3,P_5\rangle$, where $P_1,P_2 \in \pi_{AB}, P_3,P_4 \in \pi_{BC}$ and $P_5,P_6 \in \pi_{AC}$. We show that the only solids of $\mathcal{S}$ in $\langle A,B \rangle$ are $A,B$ and $C$.

 \begin{figure}[h]
 \centering
	\definecolor{rvwvcq}{rgb}{0.08235294117647059,0.396078431372549,0.7529411764705882}
	\definecolor{ffqqqq}{rgb}{1,0,0}
	\definecolor{qqwuqq}{rgb}{0,0.39215686274509803,0}
	\definecolor{ffzztt}{rgb}{1,0.6,0.2}
	\begin{tikzpicture}[line cap=round,line join=round,>=triangle 45,x=1cm,y=1cm,scale=0.4]
	\clip(-14.8148293340618,-4.4090286207433795) rectangle (12.907239857010698,12.99016200957834);
	\draw [line width=0.8pt] (-3,0)-- (-8,-2);
	\draw [line width=0.8pt] (3,0)-- (8,-2);
	\draw [shift={(0,-8.75)},line width=0.8pt]  plot[domain=1.240498971965643:1.9010936816241504,variable=\t]({1*9.25*cos(\t r)+0*9.25*sin(\t r)},{0*9.25*cos(\t r)+1*9.25*sin(\t r)});
	\draw [line width=0.8pt] (0,9)-- (0,4.89);
	\draw [shift={(12.184695481335961,-8.270687622789788)},line width=0.8pt]  plot[domain=2.185212748834585:2.8403799855141356,variable=\t]({1*21.136306558546917*cos(\t r)+0*21.136306558546917*sin(\t r)},{0*21.136306558546917*cos(\t r)+1*21.136306558546917*sin(\t r)});
	\draw [shift={(-19.076615999999994,13.2282)},line width=0.8pt]  plot[domain=5.594682347837613:5.871118964360154,variable=\t]({1*20.819290507878883*cos(\t r)+0*20.819290507878883*sin(\t r)},{0*20.819290507878883*cos(\t r)+1*20.819290507878883*sin(\t r)});
	\draw [shift={(20.656043478260884,14.197173913043487)},line width=0.8pt]  plot[domain=3.564927750876754:3.818826858464239,variable=\t]({1*22.656028302053326*cos(\t r)+0*22.656028302053326*sin(\t r)},{0*22.656028302053326*cos(\t r)+1*22.656028302053326*sin(\t r)});
	\draw [shift={(-9.667822445561141,-6.440234505862649)},line width=0.8pt]  plot[domain=0.2462183601117205:1.0113742127191454,variable=\t]({1*18.217234489211787*cos(\t r)+0*18.217234489211787*sin(\t r)},{0*18.217234489211787*cos(\t r)+1*18.217234489211787*sin(\t r)});
	\draw [shift={(0,29.5)},line width=0.8pt]  plot[domain=4.463678991291166:4.961098969478212,variable=\t]({1*32.5*cos(\t r)+0*32.5*sin(\t r)},{0*32.5*cos(\t r)+1*32.5*sin(\t r)});
	\draw [line width=0.8pt,color=ffzztt] (0,9)-- (3.684301815982976,-0.27372072639319034);
	\draw [line width=0.8pt,color=ffzztt] (3.684301815982976,-0.27372072639319034)-- (-6.127184906123695,-1.250873962449478);
	\draw [line width=0.8pt,color=ffzztt] (-6.127184906123695,-1.250873962449478)-- (0,9);
	\draw [line width=0.8pt,color=qqwuqq] (0,9)-- (6.435577461551393,-1.3742309846205571);
	\draw [line width=0.8pt,color=qqwuqq] (6.435577461551393,-1.3742309846205571)-- (-3.451434474355194,-0.18057378974207766);
	\draw [line width=0.8pt,color=qqwuqq] (-3.451434474355194,-0.18057378974207766)-- (0,9);
	\draw [line width=0.8pt,color=ffqqqq] (0,6.54462447756822)-- (3.684301815982976,-0.27372072639319034);
	\draw [line width=0.8pt,color=ffqqqq] (3.684301815982976,-0.27372072639319034)-- (-3.451434474355194,-0.18057378974207766);
	\draw [line width=0.8pt,color=ffqqqq] (-3.451434474355194,-0.18057378974207766)-- (0,6.54462447756822);
	\draw [shift={(-1.0364662961582551,-0.8721952841998019)},line width=0.8pt,color=rvwvcq]  plot[domain=-0.06708770849682377:1.4319501194002715,variable=\t]({1*7.488890289063171*cos(\t r)+0*7.488890289063171*sin(\t r)},{0*7.488890289063171*cos(\t r)+1*7.488890289063171*sin(\t r)});
	\draw [shift={(2.1611714382086733,-1.4597376420489008)},line width=0.8pt,color=rvwvcq]  plot[domain=1.8345074236957246:3.1163983355940768,variable=\t]({1*8.290987572496439*cos(\t r)+0*8.290987572496439*sin(\t r)},{0*8.290987572496439*cos(\t r)+1*8.290987572496439*sin(\t r)});
	\draw [line width=0.8pt,color=rvwvcq] (-6.127184906123695,-1.250873962449478)-- (6.435577461551393,-1.3742309846205571);
	\begin{scriptsize}
	\draw[color=black] (-6.1164825675627545,4.651638940254569) node {\normalsize{$A$}};
	\draw[color=black] (5.366586724901813,5.778043557499055) node {\normalsize{$B$}};
	\draw[color=black] (0.1100318444275692,-3.890262740516124) node {\normalsize{$C$}};
	\draw [fill=black] (0,9) circle (1pt);
	\draw[color=black] (0.31905496444732935,9.4) node {\normalsize{$P_1$}};
	\draw [fill=black] (3.684301815982976,-0.27372072639319034) circle (1pt);
	\draw[color=black] (4.0211589876375715,0.1) node {\normalsize{$P_3$}};
	\draw [fill=black] (-6.127184906123695,-1.250873962449478) circle (1pt);
	\draw[color=black] (-6.60,-0.8) node {\normalsize{$P_6$}};
	\draw [fill=black] (6.435577461551393,-1.3742309846205571) circle (1pt);
	\draw[color=black] (6.80,-0.85) node {\normalsize {$P_4$}};
	\draw [fill=black] (-3.451434474355194,-0.18057378974207766) circle (1pt);
	\draw[color=black] (-3.725,0.2) node {\normalsize{$P_5$}};
	\draw [fill=black] (0,6.54462447756822) circle (1pt);
	\draw[color=black] (0.32905496444732935,7.0) node {\normalsize{$P_2$}};
	\draw[color=black] (-8,-2.5) node {\normalsize{$\pi_{AC}$}};
\draw[color=black] (8,-2.5) node {\normalsize {$\pi_{BC}$}};
\draw[color=black] (0.03892044182375051, 10.25) node {\normalsize{$\pi_{AB}$}};
	\end{scriptsize}
	\end{tikzpicture}

\caption{There are three elements $A,B,C$ in $\mathcal{S}$ with $A\cap B\cap C=\emptyset$ and $\dim(\alpha) = 5$} 	\label{alpha5}
	\end{figure}
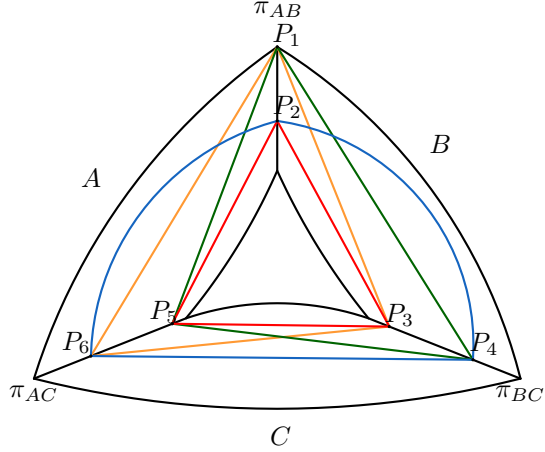
	
	\newpage
	
\begin{prop}\emph{[Using Notation \ref{notk}]}
If $\mathcal{S}$ contains three solids $A,B,C$, with $A\cap B\cap C=\emptyset$, $\dim(\alpha)=5$, and so that there are exactly four intersection planes $D'$, see Lemma \ref{alpha5space}$(i)$,  then the only solids of $\mathcal{S}$ in $\langle A,B \rangle$ are $A,B$ and $C$.
\end{prop}
\begin{proof}
Let $P_1,\dots, P_6$ be 
the intersection points of $D\cap \langle A,B \rangle$, $D\in \mathcal{S}'$, with the lines $\pi_{AB},\pi_{AC},\pi_{BC}$, and let $E$ be a solid in $\langle A,B \rangle$ different from $A,B,C$. The solid $E$ cannot contain all the points $P_1,\dots, P_6$, by its dimension so we can suppose that $P_1 \notin E$. We will first show that $E$ contains the point $P_2$.
As $E$ has a line in common with every plane intersection $D'=D\cap \langle A,B \rangle$, with $D\in \mathcal{S}'$, $E$ has at least a point in common with every line of these planes $D'$. This implies that $E$ has at least a point in common with $P_1P_3,P_1P_4, P_1P_5, $ and $P_1P_6$ or equivalently, a line $l_A$ in common with $\langle P_1,\pi_{AC} \rangle$ and a line $l_B$ in common
with $\langle P_1,\pi_{BC} \rangle$. 
Hence, $E=\langle l_A,l_B\rangle$ and so $E\subset \langle P_1,C \rangle$. If $P_2 \notin E$ then we find by symmetry that $E\subset \langle P_2,C \rangle$, and so that $E\subseteq \langle P_1,C \rangle \cap \langle P_2,C \rangle$ and  $E= C$, a contradiction. Then $P_2 \in E$; furthermore $E$ cannot contain $P_2,\dots, P_6$, by the dimension, and so we can suppose that $P_6 \notin E$. Then, by the previous arguments and symmetry, we know that $P_5$ lies in $E$. In $A$, the solid $E$ needs an extra point $P$ of $P_1P_6$ since $E$ shares a line with $\langle P_1,P_3,P_6\rangle$. This gives that $E$ contains the plane $\gamma = \langle P,P_2,P_5 \rangle$ of $A$. As $E$ also needs at least a point of each line $P_1P_3,P_1P_4$, $E$ needs at least one extra line, disjoint to $\gamma$. This gives the contradiction, again by the dimension, and so $E$ cannot be different from $A, B, C$.
\end{proof}\\

There are at most $4\cdot \bigl (\qbin{3}{1}- \qbin{2}{1} \bigr)$ solids in $\mathcal{S}'$. The first factor of this number follows since every solid in $\mathcal{S}'$ meets $\langle A,B \rangle$ in one of the four intersection planes. The second factor follows as each of these intersection planes is contained in at most $\qbin{3}{1}- \qbin{2}{1}$ solids of $\mathcal{S}'$: any two solids, intersecting $\langle A,B\rangle$ in different intersection planes, have to intersect in at least a point $Q$ outside of $\langle A,B\rangle$. There are only $3$ solids, $A,B,C$, in $\langle A,B\rangle$.

\noindent
\subsubsection*{\textbf{Case 2. Every intersection plane $D'$ contains at least two of the points $P,Q,R$}}\label{sectionbla}

\vspace{0.3cm}

Note that in this situation we have at least the red, green and blue plane (see Figure \ref{figuur52}) as intersection planes $D'=D\cap \langle A,B \rangle, D\in \mathcal{S}'$. In the following proposition, we prove how the solids in $\langle A,B\rangle$ lie with respect to the points $P,Q,R$.

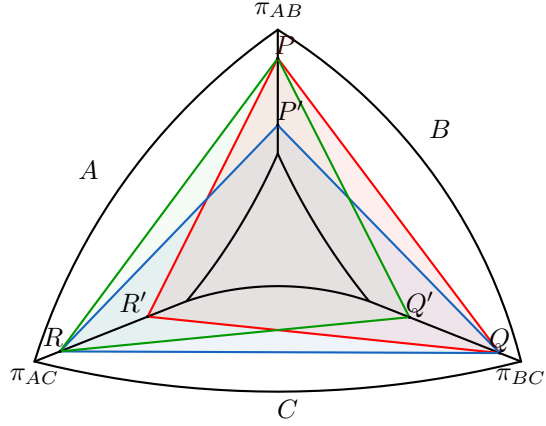
\begin{figure}[h!]
\centering
\definecolor{qqzzqq}{rgb}{0,0.6,0}
\definecolor{rvwvcq}{rgb}{0.08235294117647059,0.396078431372549,0.7529411764705882}
\definecolor{ffqqqq}{rgb}{1,0,0}
\definecolor{ttqqqq}{rgb}{0.2,0,0}
\begin{tikzpicture}[line cap=round,line join=round,>=triangle 45,x=1cm,y=1cm, scale=0.4]
\clip(-12.618909221655894,-5) rectangle (12.981195715718668,12.318870368998288);
\fill[line width=0.8pt,color=ffqqqq,fill=ffqqqq,fill opacity=0.07] (7.297741915189411,-1.7190967660757643) -- (-4.278288438424419,-0.5113153753697677) -- (0,8.052186212769172) -- cycle;
\fill[line width=0.8pt,color=rvwvcq,fill=rvwvcq,fill opacity=0.07] (7.297741915189411,-1.7190967660757643) -- (-7.161978419875809,-1.6647913679503237) -- (0,5.8335104515300324) -- cycle;
\fill[line width=0.8pt,color=qqzzqq,fill=qqzzqq,fill opacity=0.04] (-7.161978419875809,-1.6647913679503237) -- (4.32577570981604,-0.530310283926416) -- (0,8.052186212769172) -- cycle;
\draw [line width=0.8pt] (-3,0)-- (-8,-2);
\draw [line width=0.8pt] (3,0)-- (8,-2);
\draw [shift={(0,-8.75)},line width=0.8pt]  plot[domain=1.240498971965643:1.9010936816241504,variable=\t]({1*9.25*cos(\t r)+0*9.25*sin(\t r)},{0*9.25*cos(\t r)+1*9.25*sin(\t r)});
\draw [line width=0.8pt] (0,9)-- (0,4.89);
\draw [shift={(12.184695481335961,-8.270687622789788)},line width=0.8pt]  plot[domain=2.185212748834585:2.8403799855141356,variable=\t]({1*21.136306558546917*cos(\t r)+0*21.136306558546917*sin(\t r)},{0*21.136306558546917*cos(\t r)+1*21.136306558546917*sin(\t r)});
\draw [shift={(-19.076615999999994,13.2282)},line width=0.8pt]  plot[domain=5.594682347837613:5.871118964360154,variable=\t]({1*20.819290507878883*cos(\t r)+0*20.819290507878883*sin(\t r)},{0*20.819290507878883*cos(\t r)+1*20.819290507878883*sin(\t r)});
\draw [shift={(20.656043478260884,14.197173913043487)},line width=0.8pt]  plot[domain=3.564927750876754:3.818826858464239,variable=\t]({1*22.656028302053326*cos(\t r)+0*22.656028302053326*sin(\t r)},{0*22.656028302053326*cos(\t r)+1*22.656028302053326*sin(\t r)});
\draw [shift={(-9.667822445561141,-6.440234505862649)},line width=0.8pt]  plot[domain=0.2462183601117205:1.0113742127191454,variable=\t]({1*18.217234489211787*cos(\t r)+0*18.217234489211787*sin(\t r)},{0*18.217234489211787*cos(\t r)+1*18.217234489211787*sin(\t r)});
\draw [shift={(0,29.5)},line width=0.8pt]  plot[domain=4.463678991291166:4.961098969478212,variable=\t]({1*32.5*cos(\t r)+0*32.5*sin(\t r)},{0*32.5*cos(\t r)+1*32.5*sin(\t r)});
\draw [line width=0.8pt,color=ffqqqq] (7.297741915189411,-1.7190967660757643)-- (-4.278288438424419,-0.5113153753697677);
\draw [line width=0.8pt,color=ffqqqq] (-4.278288438424419,-0.5113153753697677)-- (0,8.052186212769172);
\draw [line width=0.8pt,color=ffqqqq] (0,8.052186212769172)-- (7.297741915189411,-1.7190967660757643);
\draw [line width=0.8pt,color=rvwvcq] (7.297741915189411,-1.7190967660757643)-- (-7.161978419875809,-1.6647913679503237);
\draw [line width=0.8pt,color=rvwvcq] (-7.161978419875809,-1.6647913679503237)-- (0,5.8335104515300324);
\draw [line width=0.8pt,color=rvwvcq] (0,5.8335104515300324)-- (7.297741915189411,-1.7190967660757643);
\draw [line width=0.8pt,color=qqzzqq] (-7.161978419875809,-1.6647913679503237)-- (4.32577570981604,-0.530310283926416);
\draw [line width=0.8pt,color=qqzzqq] (4.32577570981604,-0.530310283926416)-- (0,8.052186212769172);
\draw [line width=0.8pt,color=qqzzqq] (0,8.052186212769172)-- (-7.161978419875809,-1.6647913679503237);
\begin{scriptsize}
\draw[color=black] (-8,-2.5) node {\normalsize{$\pi_{AC}$}};
\draw[color=black] (8,-2.5) node {\normalsize{$\pi_{BC}$}};
\draw[color=black] (0.03892044182375051,9.659303911615472) node {\normalsize{$\pi_{AB}$}};
\draw[color=black] (-6.218882987312253,4.340170996849841) node {\normalsize{$A$}};
\draw[color=black] (5.329608795547827,5.705509926843159) node {\normalsize{$B$}};
\draw[color=black] (0.3233660522390234,-3.567416972694787) node {\normalsize{$C$}};
\draw [fill=black] (-7.161978419875809,-1.6647913679503237) circle (1pt);
\draw[color=black] (-7.385109990014873,-1.178073845206482) node {\normalsize{$R$}};
\draw [fill=ttqqqq] (0,8.052186212769172) circle (1pt);
\draw[color=ttqqqq] (0.23803236911444153,8.49307690891285) node {\normalsize{$P$}};
\draw [fill=black] (7.297741915189411,-1.7190967660757643) circle (1pt);
\draw[color=black] (7.3,-1.291852089372592) node {\normalsize{$Q$}};
\draw [fill=black] (-4.278288438424419,-0.5113153753697677) circle (1pt);
\draw[color=black] (-4.739765813152835,-0.01184684250385698) node {\normalsize{$R'$}};
\draw [fill=black] (0,5.8335104515300324) circle (1pt);
\draw[color=black] (0.40869973536360527,6.331290269756762) node {\normalsize{$P'$}};
\draw [fill=black] (4.32577570981604,-0.530310283926416) circle (1pt);
\draw[color=black] (4.675383891592699,-0.09718052562843926) node {\normalsize{$Q'$}};
\end{scriptsize}
\end{tikzpicture}
 \caption{There are three elements $A,B,C$ in $\mathcal{S}$ with $A\cap B\cap C=\emptyset$ and $\dim(\alpha) = 5$} \label{figuur52}
\end{figure}

\begin{prop}\emph{[Using Notation \ref{notk}]}
If $\mathcal{S}$ contains three solids $A,B,C$, with $A\cap B\cap C=\emptyset$, $\dim(\alpha)=5$, and so that every intersection plane $D'$ contains at least two of the points $P,Q,R$, see Lemma \ref{alpha5space}$(ii)$, then all the solids of $\mathcal{S}$ in $\langle A,B \rangle$, also contain at least two of the points $P,Q,R$.
\end{prop}
\begin{proof}
Let $E$ be a solid of $\mathcal{S}$ in $\langle A, B \rangle$, different from $A,B$ and $C$. Suppose $P\notin E$, then we have to prove that $E$ contains the points $R$ and $Q$. We find that $E\cap A$ and $E\cap B$ are subspaces that meet the lines $PR$, $PR'$, $P'R$ and $PQ,PQ', P'Q$, respectively, as $E$ meets every intersection plane $D'$ in at least a line. Hence, $E$ meets $A$ in a line $l_{AE}$ through $R$ and a point of $PR'$, or $E$ has a plane $\gamma_{AE}$ in common with $A$. By symmetry, $E$ meets $B$ in a line $l_{BE}$ through $Q$ and a point of $PQ'$, or $E$ has a plane $\gamma_{BE}$ in common with $B$. 
\begin{itemize}
    \item [$a)$] If $\dim(A\cap E)=\dim(B\cap E)=2$, then the planes $\gamma_{AE}$ and $\gamma_{BE}$ meet in a point of $\pi_{AB}$ as they cannot contain the line $\pi_{AB}$ since $P\notin E$. Hence, $E$ contains two planes meeting in a point, which gives a contradiction since $\dim(E)=3$.
    \item [$b)$] If $\dim(A\cap E)=2$ and $\dim(B\cap E)=1$, then $\gamma_{AE}\cap \pi_{AB}=l_{BE} \cap \pi_{AB}$. First note that $l_{BE} \cap \pi_{AB}$ is not empty by the dimension of $E$. Now, if  $\gamma_{AE}\cap \pi_{AB}\not =l_{BE} \cap \pi_{AB}$, then $\pi_{AB}\subset E$, which gives a contradiction as $P\notin E$. Since $l_{BE}$ can only meet $\pi_{AB}$ in the point $P$, we find a contradiction, again as $P\notin E$.
\end{itemize}
Hence we know
that $E$ contains a line $l_{AE} \subset A$ through $R$ and a line $l_{BE} \subset B$ through $Q$, which proves the proposition.
\end{proof}
\\

There are at most $\bigl (3 \cdot \qbin{2}{1}-2 \bigr ) \bigl (\qbin{3}{1}- \qbin{2}{1} \bigr )$ solids not in $\langle A,B\rangle$. This follows as two solids $D_1,D_2$, intersecting $\langle A,B\rangle$ in the intersection planes $D'_1$ and $D'_2$ meeting in a point, then $D_1$ and $D_2$ have to intersect in at least a point not in $\langle A,B\rangle$. And there are at most $3 \cdot \qbin{2}{1}-2$ intersection planes $D'$. There are at most $\qbin{3}{1}+3q \qbin{3}{1}$ solids in $\langle A,B\rangle$, namely all the solids through the plane $\langle P,Q,R \rangle$ and all solids through precisely two of the three points $P,Q,R$ in $\langle A,B \rangle$.

\begin{remark}
Note that if $\mathcal{S}$ contains three elements $A,B,C,$ with $A\cap B\cap C=\emptyset$, and if $\dim(\alpha)=5$, then the number of elements of $\mathcal{S}$ is at most $f(3,q)=3q^4+6q^3+5q^2+q+1$, and so we will not consider these maximal sets of solids in our classification.
\end{remark}
\subsubsection{General case $k>3$ and $\alpha$ is a $(k+2)$-space}
In this case we prove that all the $k$-spaces of $\mathcal{S}$ contain $\pi_{ABC}$. This implies that we can investigate this case by considering the quotient space of $\pi_{ABC}$ and use the previous results for $k=3$. 
\begin{prop}\emph{[Using Notation \ref{notk}]} If $\mathcal{S}$ contains three $k$-spaces $A,B,C$,  with $\dim(A\cap B\cap C)=k-4$,  and if $\dim(\alpha)=k+2$, then every $k$-space in $\mathcal{S}$ contains $\pi_{ABC}$.
\end{prop}
\begin{proof}
By Lemma \ref{lemmahoespaceserbuiten}, we know that all the $k$-spaces of $\mathcal{S}$ outside of $\langle A,B \rangle$ contain $\pi_{ABC}$. It is also clear that $A,B$ and $C$ contain $\pi_{ABC}$.\\
Suppose that there is a $k$-space $E$ in $\langle A,B \rangle$, not through $\pi_{ABC}$. As $E$ has to meet all the $(k-1)$-spaces $D'_i$ in at least a $(k-2)$-space, $E$ has to meet $\pi_{ABC}$ in a $(k-5)$-space,  $\pi_{AB} \setminus \pi_{ABC}$ in a line,  $\pi_{BC} \setminus \pi_{ABC}$ in a line and  $\pi_{AC} \setminus \pi_{ABC}$ in a line. This would imply
that $\dim(E)=k+1$, which gives a contradiction.
\end{proof}
 \\

Clearly, the previous proposition implies that in order to have an estimate of the number of $k$-spaces in and outside of $\langle A, B \rangle$, we can use the results for $k=3$ in Section \ref{sectie5}: $|\mathcal{S}|\leq 4\cdot \bigl (\qbin{3}{1}- \qbin{2}{1} \bigr) +3$ or $|\mathcal{S}|\leq \bigl (3 \cdot \qbin{2}{1} -2 \bigr ) \bigr (\qbin{3}{1}- \qbin{2}{1} \bigr ) + \qbin{3}{1}\bigr (3q^2+1 \bigr )$. In both cases, $|\mathcal{S}|<\theta_{k+1}+q^4+2q^3+3q^2=f(k,q)$.
\\

To conclude this section we give a theorem which summarizes the cases studied in this section.

\begin{prop}\label{overzicht1}\emph{[Using Notation \ref{notk}]}
In the projective space $\PG(n,q)$, with $n \geq k+2$ and $k\geq3$, let $\mathcal{S}$ be a maximal set of $k$-spaces pairwise intersecting in at least a $(k-2)$-space such that $\mathcal{S}$ contains three $k$-spaces
$A,B,C$,  with $\dim(A\cap B\cap C)=k-4$,  and such that $|\mathcal{S}|\geq f(k,q)$. Then we have one of the following possibilities:
\begin{itemize}
    \item [i)] there are no $k$-spaces of $\mathcal{S}$ outside of $\langle A,B \rangle$ and $\mathcal{S}$ is Example \ref{voorbeeldenalgemeen}$(x)$, 
    \item [ii)]  $\dim(\alpha)=k-1$ and $\mathcal{S}$ is Example \ref{voorbeeldenalgemeen}$(v)$, 
    \item [iii)]  $\dim(\alpha)=k$ and $\mathcal{S}$ is Example \ref{voorbeeldenalgemeen}$(iv)$, 

    \item [iv)]  $\dim(\alpha)=k+1$ and $\mathcal{S}$ is Example \ref{voorbeeldenalgemeen}$(vi)$. 
    

\end{itemize}
\end{prop}

\section{Every three elements of $\mathcal{S}$ meet in at least a $(k-3)$-space }\label{noconfi}
Throughout this section we suppose that every three elements of $\mathcal{S}$ meet in at least a $(k-3)$-space. Moreover, to avoid trivial cases, we can suppose that there exist two $k$-spaces in $\mathcal{S}$ intersecting in precisely a $(k-2)$-space. We can find those two $k$-spaces as otherwise all subspaces would pairwise intersect in a $(k-1)$-space and the classification in this case is known: all the $k$-spaces go through a fixed $(k-1)$-space or all the $k$-spaces lie in a $(k+1)$-dimensional space, see Theorem \ref{basisEKRthm}. We also suppose that $\mathcal{S}$ is not a $(k-2)$- or a $(k-3)$-pencil as in this case either $\mathcal{S}$ is Example \ref{voorbeeldenalgemeen}$(i)$ or we can investigate the quotient space and use the known Erd{\H{o}}s-Ko-Rado results \cite{EKRplanes}. We begin this section with a useful lemma.

\begin{lemma}\label{lemmaonZZ}
Let $\mathcal{S}$ be a maximal set of $k$-spaces in $\PG(n,q)$ pairwise intersecting in at least a $(k-2)$-space such that for every $X,Y,Z \in \mathcal{S}$, $\dim(X\cap Y \cap Z) \geq k-3$, and such that there is no point contained in all elements of $\mathcal{S}$. Then there exist three elements $A,B,C$ of $\mathcal{S}$  such that 
\begin{itemize}
    \item [$a)$] $\pi=A \cap B\cap C$ is a $(k-3)$-space, 
    \item [$b)$] at least two of the three subspaces $\pi_{AB}=A\cap B,\pi_{BC}=B\cap C,\pi_{AC}=A\cap C$ have dimension $k-2$, and at most one of them has dimension $k-1$.
    \item [$c)$] $\zeta=\langle \pi_{AB},\pi_{BC},\pi_{AC}\rangle$ has dimension $k$ or $k+1$.
\end{itemize}
Every $k$-space in $\mathcal{S}$ not through $\pi$ meets the space $\zeta=\langle \pi_{AB},\pi_{BC},\pi_{AC}\rangle$ in at least a $(k-1)$-space.

\end{lemma}

\proof If every three $k$-spaces in $ \mathcal{S}$ meet (at least) in a $(k-2)$-space,
then $\mathcal{S}$ is a $(k-2)$-pencil, and so there is a point contained in all the $k$-spaces of $\mathcal{S}$. Therefore, there exist three elements $A,B,C \in \mathcal{S}$ such that $\pi=A\cap B\cap C$ is a $(k-3)$-space.
Let $\pi_{AB}=A\cap B$, $\pi_{BC}=B\cap C $ and $\pi_{AC}=A\cap C$, and let $\zeta=\langle \pi_{AB},\pi_{BC},\pi_{AC}\rangle$. Note that at least two of the three subspaces $\pi_{AB},\pi_{BC},\pi_{AC}$ have dimension $k-2$. Otherwise, if, for example, $\dim(\pi_{AB})=\dim(\pi_{AC})=k-1$, then the $k$-space $A$ contains two $(k-1)$-spaces, $\pi_{AB}$ and $\pi_{AC}$, meeting in a $(k-3)$-space, a contradiction. W.l.o.g. we can suppose that $\dim(\pi_{AB})=\dim(\pi_{AC})=k-2$ and $\dim(\pi_{BC})\in \{k-1,k-2\}.$ This also implies that the dimension of $\zeta$ is at most $k+1$. 
On the other hand, note that $\zeta$ has at least dimension $k$. Otherwise, if $\zeta=\langle \pi_{AB},\pi_{BC},\pi_{AC}\rangle$ is a
$(k-1)$-space.  then $\zeta=\langle \pi_{AB},\pi_{AC}\rangle$ and so $\zeta \subset A$. By the same argument, $\zeta \subset B$, and $\zeta \subset C$. Hence $\zeta\subset A  \cap B \cap C=\pi$, a contradiction.\\
\textcolor{black}{\textsc{Case 1.} Suppose that $\pi_{AB},\pi_{AC}$ and $\pi_{BC}$ are $(k-2)$-spaces. Then, $\zeta$ is a $k$-space and} consider a $k$-space $G$ in $\mathcal{S}$ not through $\pi$. This $k$-space exists since there is no point contained in all elements of $\mathcal{S}$, and hence not all elements of $\mathcal{S}$ contain $\pi$.  Then $G$ meets $\pi$ in a $(k-4)$-space $\pi_G$ and it contains at least a $(k-3)$-space of $\pi_{AB},\pi_{BC}$ and $\pi_{AC}$. This follows since any three elements of $\mathcal{S}$ meet in at least  a $(k-3)$-space  and $\pi \nsubseteq G$. 
Since the three subspaces $G\cap \pi_{AB},G\cap \pi_{BC}$ and $G\cap \pi_{AC}$ have dimension at least $k-3$, since they pairwise meet in the $(k-4)$-space $\pi_G$, and since
$\pi_{AB},\pi_{AC}$ and $\pi_{BC}$ span at least a $k$-space, $G$ contains the subspace $\langle G\cap \pi_{AB},G\cap \pi_{BC},G\cap \pi_{AC}\rangle $, with at least dimension  $k-1$, in $\zeta$.\\
\textcolor{black}{\textsc{Case 2.}
Suppose that  $\dim(\pi_{AB})=\dim(\pi_{AC})=k-2$ and $\dim(\pi_{BC})=k-1$. They meet in  the $(k-3)$-space $\pi$. Now, $\zeta$ is a $(k+1)$-space and consider a $k$-space $G$ not through $\pi$. As before $G$ meets $\pi$ in a $(k-4)$-space; the spaces $G \cap \pi_{AB}$ and $G \cap \pi_{AC}$ are $(k-3)$-spaces  otherwise $G$ goes through $\pi$ and finally  $\dim(G \cap \pi_{BC})\in\{k-3,k-2\}$.\\
\textsl{Case 2a.} $\dim(G \cap \pi_{BC})=k-3$. Then $G \cap \pi_{AC}$ and $G \cap \pi_{BC}$ cannot be contained in $\pi_{AB}$ otherwise $\dim(G \cap \pi)=k-3$. Hence, the latter and $G \cap \pi_{AB}$ are linearly independent spaces (i.e. the span of two of them not meet the other space) $(k-3)$-spaces pairwise intersecting in $G \cap \pi$. Therefore
$$\dim\langle\pi_{AB} \cap G, \pi_{AC} \cap G, \pi_{BC} \cap G\rangle=k-1.$$
\textsl{Case 2b.} $\dim(G \cap \pi_{BC})=k-2$. Note that $G \cap \pi_{BC}$ cannot meet $\pi_{AB}$ in a $k-3$ space, otherwise $G$ goes through $\pi$. Then, again $G \cap \pi_{XY}$ with $\{X,Y\}\subset \{A,B,C\}$ are linearly independent spaces $(k-3)$-spaces pairwise intersecting in $G \cap \pi$ and $$\dim\langle\pi_{AB} \cap G, \pi_{AC} \cap G, \pi_{BC} \cap G\rangle=
k.$$
Hence, the $k$-space $G$ is inside of $\zeta$.\\
So, in any case, we get that  a $k$-space not through $\pi$ meets $\zeta$ in at least a $(k-1)$-space. 
}
\endproof

\begin{theorem}\label{thmnoconfigalgemeenk}
Let $\mathcal{S}$ be a maximal set of $k$-spaces pairwise intersecting in at least a $(k-2)$-space in $\PG(n,q)$.  If for every three elements $X,Y,Z$ of $\mathcal{S}$: $\dim(X\cap Y\cap Z)\geq k-3$, and if there is no point contained in all elements of $\mathcal{S}$, then $\mathcal{S}$ is one of the following examples:
\begin{itemize}
    \item [(i)] Example \ref{voorbeeldenalgemeen}$(ii)$: Star.
    \item[(ii)] Example \ref{voorbeeldenalgemeen}$(iii)$: Generalized Hilton-Milner example.
\end{itemize}
\end{theorem}

\begin{proof}
From Lemma \ref{lemmaonZZ}, it follows that we can suppose that there are three $k$-spaces $A,B,C$ with $\dim(A\cap B\cap C)=k-3$, $\dim(\pi_{AB})=\dim(\pi_{AC})=k-2$ and $\dim(\pi_{BC})\in \{k-1,k-2\}$

\textsc{Case 1.} \emph{$\dim(\pi_{BC})=k-2$.}
In this case we know, again from Lemma \ref{lemmaonZZ}, that $\zeta=\langle \pi_{AB},\pi_{AC}, \pi_{BC}\rangle$ has dimension $k$ and that any element of $\mathcal{S}$, not through $\pi=A\cap B\cap C$, meets $\zeta$ in at least a $(k-1)$-space. 

\emph{Case 1.1. Suppose that there exists a $k$-space $D$, not containing $\pi$, with $\dim(D\cap A)=\dim(D\cap B)=\dim(D\cap C)=k-2$}. Let $\pi_{AD}, \pi_{BD}$ and $\pi_{CD}$ be these $(k-2)$-spaces. Note that each of them contains the $(k-4)$-space $\pi_D=D\cap \pi$ and that they are contained in $\zeta$. We prove that all elements of $\mathcal{S}$ meet $\zeta$ in at least a $(k-1)$-space. From Lemma \ref{lemmaonZZ}, it follows that we only have to check that all elements of $\mathcal{S}$ through $\pi$ have this property. 
 Let $E$ be a $k$-space in $\mathcal{S}$ through $\pi$. Then $E$ contains a $(k-3)$-space of  $\pi_{AD}, \pi_{BD}$ and $\pi_{CD}$. At least two of these $(k-3)$-spaces are different, { since $\pi$ is not contained in $D$,} and span together with $\pi$ at least a $(k-1)$-space contained in the $k$-space $\zeta$. Hence every $k$-space of $\mathcal{S}$ meets $\zeta$ in at least a $(k-1)$-space. Then $\mathcal{S}$ is Example \ref{voorbeeldenalgemeen}$(ii)$.
 
\emph{Case 1.2. There exists no $k$-space $D$ in $\mathcal{S}$, not containing $\pi$, with $\dim(D\cap A)=\dim(D\cap B)=\dim(D\cap C)=k-2$}. 

 In this case we will prove that if not every $k$-space of $\mathcal{S}$ meets $\zeta$ in a $(k-1)$-space, that then $\mathcal{S}$ is the second example described in the theorem. 
Let $D$ be a $k$-space of $\mathcal{S}$ not containing $\pi$ and meeting $A,B$ or $C$ in a $(k-1)$-space. W.l.o.g. we can suppose that $C\cap D$ is the $(k-1)$-space $\pi_{CD}$ and that $A\cap D$ and $B\cap D$ are $(k-2)$-spaces ($\pi_{AD}$ and $\pi_{BD}$ respectively). Note that these subspaces $\pi_{AD}, \pi_{BD}, \pi_{CD}$ contain the
$(k-4)$-space $\pi_D=D\cap \pi$ and that $\pi_{AD}, \pi_{BD}\subset \zeta$. {This follows since $D$ meets $\pi_{AB}, \pi_{AC},\pi_{BC}$ in a $(k-3)$-space, and $D\cap \pi_{AB}$ and $D \cap \pi_{AC}$ span $\pi_{AD}$. The same argument holds for the space $B$.} Suppose that $\mathcal{S}$ is not a Star, and so, suppose that $F \in \mathcal{S}$ is a $k$-space that meets $\zeta$ in (at most) a $(k-2)$-space. As every $k$-space in $\mathcal{S}$, not containing $\pi$, meets $\zeta$ in a $(k-1)$-space (Lemma \ref{lemmaonZZ}), we see that $F$ contains $\pi$. Now, since every three elements of $\mathcal{S}$ meet in a $(k-3)$-space, $F$ also contains a $(k-3)$-space of the two $(k-2)$-spaces $\pi_{AD}$ and $\pi_{BD}$ in $\zeta$ 
($\pi_{ADF}$, $\pi_{BDF}$ respectively). 
As $F$ has no $(k-1)$-space in common with $\zeta$, and since $\pi_{AD},\pi_{BD}\subset \zeta$, $\pi_{CD}\nsubseteq \zeta$, 
we find that $\pi_{ADF}=\pi_{BDF}=\pi_{AB}\cap D$ and that $\pi_{CDF}\nsubseteq \zeta$. Hence, $F\cap \zeta=\pi_{AB}$ and $C\cap F =\langle \pi_{CDF}, \pi\rangle$.
Let $\nu=\langle \zeta,C\rangle$. Then we prove that every $k$-space in $\mathcal{S}$ is contained in $\nu$ or contains $\pi_{AB}$ and meets $\nu$ in a $(k-1)$-space. 
Every $k$-space in $\mathcal{S}$ containing $\pi_{AB}$ must contain at least a $(k-2)$-space of $C$. Hence, this $k$-space meets $\nu$ in at least a $(k-1)$-space. 
Consider now a $k$-space $E\in \mathcal{S}$ not through $\pi_{AB}$. From the arguments above it follows that, if $\pi\subset E$, then $E\subset \nu$. Indeed, if $\pi \not \subseteq E$, then, by Lemma \ref{lemmaonZZ}, $E$ contains a $(k-1)$-space in $\zeta$ and a point in $C \setminus \zeta$ as otherwise we have \emph{Case $1.1$}, and so $\mathcal{S}$ would be a Star, a contradiction. Hence, $E\subset \nu$.

\textsc{Case 2.} \emph{For every three $k$-spaces $ X,Y,Z \in \mathcal{S}$, we have that $ \dim(X \cap Y \cap Z) \geq k-2$ or two of these spaces meet in a $(k-1)$-space.}
Since we suppose that there is no point contained in all elements of $\mathcal{S}$, we see that not every three elements meet in a $(k-2)$-space.
Recall that $A\cap B =\pi_{AB}$ is a $(k-2)$-space. Hence, every other element of $\mathcal{S}$ contains $\pi_{AB}$ or meets $A$ or $B$ in a $(k-1)$-space. Note that the elements of $\mathcal{S}$, not through $\pi_{AB}$, are contained in $\langle A,B \rangle$. By Example \ref{voorbeeldenalgemeen}$(x)$, we may suppose that not all elements of $\mathcal{S}$ are contained in $\langle A,B \rangle$. Hence, let $D \in  \mathcal{S}$ be a $k$-space not contained in $\langle A,B \rangle$. 

If $D\cap A=D\cap B=\pi_{AB}$ then, by symmetry, it follows that every element of $\mathcal{S}$, not through $\pi_{AB}$, meets two of the three elements $A,B,D$ in a $(k-1)$-space. This is a contradiction since a $k$-space cannot contain two $(k-1)$-spaces, meeting in a $(k-3)$-space.\\ 
Hence, every $k$-space in $\mathcal{S}$, not in $\langle A,B \rangle$, meets $A$ or $B$ in a $(k-1)$-space through $\pi_{AB}$. W.l.o.g. we suppose that $B\cap D=\pi_{BD}$ is a $(k-1)$-space, and so $A\cap D=\pi_{AD}=\pi_{AB}$. Consider now an element $E\in \mathcal{S}$ not through $\pi_{AB}$. Then, $E\subset \langle A,B \rangle$, and  since both $A,B$ and $A,D$ meet in a $(k-2)$-space, $E$ contains a $(k-1)$-space in $A$ or $E$ contains a $(k-1)$-space in both $D$ and $B$.
Note that $E$ cannot contain a $(k-1)$-space of $D$, since $E\subset \langle A,B \rangle$, but $D\cap \langle A,B \rangle$ is a $(k-1)$-space through $\pi_{AB}\nsupseteq E$.
Hence, $E$ must contain a $(k-1)$-space of $A$ and a $(k-2)$-space of   $B \cap D$ and so every element of $\mathcal{S}$, not through $\pi_{AB}$, is contained in $\nu=\langle A,\pi_{BD} \rangle$.\\
To conclude this proof, we show that every element of $\mathcal{S}$, through $\pi_{AB}$, meets $\nu=\langle A,\pi_{BD} \rangle$ in at least a $(k-1)$-space, which proves that $\mathcal{S}$ is the Generalized Hilton-Milner example.
So, consider a $k$-space $F\in \mathcal{S}$, $\pi_{AB}\subset F$. Then $F$ must contain a $(k-2)$-space $\pi_{EF}$ of $E$. Hence, $F$ contains the $(k-1)$-space $\langle \pi_{EF},\pi_{AB}\rangle \subset \langle A, \pi_{BD} \rangle$.
\end{proof}

\section{There is at least a point contained in all $k$-spaces of $\mathcal{S}$}\label{sectionsolidpointk}

To classify all maximal sets of $k$-spaces pairwise intersecting in at least a $(k-2)$-space, we also have to investigate the families of $k$-spaces such that there is a subspace contained in all its elements.\\
More precisely, in this section we will consider a set $\mathcal{S}$ of $k$-spaces of $\PG(n,q)$ such that there is at least a point contained in all elements of $\mathcal{S}$. So, let $g$, with $0 \leq g\leq k-3$, be the dimension of the maximal subspace $\gamma$ contained in all elements of $\mathcal{S}$. In the quotient space of $\PG(n,q)$ with respect to $\gamma$, the set $\mathcal{S}$ of $k$-spaces corresponds to a set $\mathcal{T}$ of $(k-g-1)$-spaces in $\PG(n-g-1,q)$ that pairwise intersect in at least a $(k-g-3)$-space, and so that there is no point contained in all elements of $\mathcal{T}$. 
Since we are interested in sets $\mathcal{S}$ of $k$-spaces with $|\mathcal{S}|>f(k,q)$, this corresponds with sets $\mathcal{T}$ of $(k-g-1)$-spaces with $|\mathcal{T}|>f(k,q)$.\\
Now, if $k-g-1>2$, we can use Theorem \ref{overzicht1} and Theorem \ref{thmnoconfigalgemeenk} for the sets $\mathcal{T}$ in $\PG(n-g-1,q)$. For each example we show that it can be extended to one of the examples discussed in the previous sections. 

\begin{enumerate}
    \item $\mathcal{T}$ is the set of $k'$-spaces of Theorem \ref{overzicht1}(i), so that $\mathcal{T}$ is Example \ref{voorbeeldenalgemeen}$(x)$ : There exists a $(k'+2)$-space $\rho'$ such that $\mathcal{T}$ is the set of all $k'$-spaces in $\rho$. Then $\mathcal{S}$ can be extended to Example \ref{voorbeeldenalgemeen}$(x)$ in $\PG(n,q)$, with $\rho=\langle \rho', \gamma \rangle$.
    \item $\mathcal{T}$ is the set of $k'$-spaces of Theorem \ref{overzicht1}(ii), so that $\mathcal{T}$ is Example \ref{voorbeeldenalgemeen}$(v)$ : There are a $(k'+2)$-space $\rho'$, and a $(k'-1)$-space $\alpha' \subset \rho'$  so that $\mathcal{T}$ contains all $k'$-spaces in $\rho'$ that meets $\alpha'$ in at least a $(k'-2)$-space, and all $k'$-spaces in $\PG(n-g-1,q)$ through $\alpha'$. Then $\mathcal{S}$ can be extended to Example \ref{voorbeeldenalgemeen}$(v)$ in $\PG(n,q)$, with $\rho=\langle \rho', \gamma \rangle$ and $\alpha=\langle \alpha', \gamma \rangle$.
    \item $\mathcal{T}$ is the set of $k'$-spaces of Theorem \ref{overzicht1}(iii), so that $\mathcal{T}$ is Example \ref{voorbeeldenalgemeen}$(iv)$ : There are a $(k'+2)$-space $\rho'$, a $k'$-space $\alpha' \subset \rho'$ and a $(k'-2)$-space $\pi'\subset \alpha'$ so that $\mathcal{T}$ contains all $k'$-spaces in $\rho'$ that meets $\alpha'$ in at least a $(k'-1)$-space,  all $k'$-spaces in $\rho'$ through $\pi'$, and all $k'$-spaces in $\PG(n-g-1,q)$ that contain a $(k'-1)$-space of $\alpha'$ through $\pi'$. Then $\mathcal{S}$ can be extended to Example \ref{voorbeeldenalgemeen}$(iv)$ in $\PG(n,q)$, with $\pi=\langle \pi', \gamma \rangle$, $\rho=\langle \rho', \gamma \rangle$ and $\alpha=\langle \alpha', \gamma \rangle$.
    \item $\mathcal{T}$ is the set of $k'$-spaces of Theorem \ref{overzicht1}(iv). Since we suppose  that $|\mathcal{S}|=|\mathcal{T}|>f(k,q)$, we know that $\mathcal{T}$ is Example \ref{voorbeeldenalgemeen}$(vi)$: There are two $(k'+2)$-spaces $\rho'_1,\rho'_2$ intersecting in a $(k'+1)$-space $\alpha'=\rho'_1\cap \rho'_2$. There are two $(k'-1)$-spaces $\pi'_A,\pi'_B\subset \alpha'$, with $\pi'_A\cap \pi'_B$ the $(k'-2)$-space $l'$, there is a point $P'\in \alpha'\setminus \langle \pi'_A,\pi'_B \rangle$, and let $P'_A, P'_B \subset l'$ be two different $(k'-3)$-spaces. Then $\mathcal{T}$ contains 
    \begin{itemize}
        \item [$\circ$] all $k'$-spaces in $\alpha'$,
        \item [$\circ$]  all $k'$-spaces through $\langle P',l' \rangle$,
        \item [$\circ$]  all $k'$-spaces in $\rho'_1$ through $P'$ and a $(k'-2)$-space in $\pi'_A$ through $P'_A$,
        \item [$\circ$] all $k'$-spaces in $\rho'_1$ through $P'$ and a $(k'-2)$-space in $\pi'_B$ through $P'_B$,
        \item [$\circ$] all $k'$-spaces in $\rho'_2$ through $P'$ and a $(k'-2)$-space in $\pi'_A$ through $P'_B$,
        \item [$\circ$] all $k'$-spaces in $\rho'_2$ through $P'$ and a $(k'-2)$-space in $\pi'_B$ through $P'_A$.
    \end{itemize}

    Then $\mathcal{S}$ can be extended to Example \ref{voorbeeldenalgemeen}$(vi)$ in $\PG(n,q)$, with $P_A=\langle P_A', \gamma \rangle$, $P_B=\langle P_B', \gamma \rangle$, $\pi_A=\langle \pi_A', \gamma \rangle$, $\pi_B=\langle \pi_B', \gamma \rangle$, $l=\langle l', \gamma \rangle$, $\alpha=\langle \alpha', \gamma \rangle$, $\rho_1=\langle \rho'_1, \gamma \rangle$ and $\rho_2=\langle \rho'_2, \gamma \rangle$.
    \item  If $\mathcal{T}$ is the set of $k'$-spaces of Proposition \ref{overzicht1}(iv), then $\mathcal{S}$ can be extended to a set $\mathcal{S}'$ of $k$-spaces pairwise intersecting in a $(k-2)$-space such that  $\mathcal{S}'$ contains three $k$-spaces that meet in a $(k-4)$-space with $\dim(\alpha)=k+2$. Hence, $|\mathcal{S}'|<f(k,q)$ and  so these sets $\mathcal{T}$ do not lead to large examples of $\mathcal{S}$.
    \item  $\mathcal{T}$ is the set of $k'$-spaces of Theorem \ref{thmnoconfigalgemeenk}(i): There exists a $k'$-space $\zeta'$ such that  $\mathcal{T}$ is the set of all $k'$-spaces that have a $(k'-1)$-space in common with $\zeta'$. Then $\mathcal{S}$ can be extended to example $(i)$ in Theorem \ref{thmnoconfigalgemeenk} with $\zeta=\langle \zeta', \gamma \rangle$.
    \item  $\mathcal{T}$ is the set of $k'$-spaces of Theorem \ref{thmnoconfigalgemeenk}(ii): There exists a $(k'+1)$-space $\nu'$ and a $(k'-2)$-space $\pi'\subset \nu$ such that $\mathcal{T}$ consists of all $k'$-spaces in $\nu'$, together with all $k'$-spaces through $\pi'$ that intersect $\nu'$ in at least a $(k'-1)$-space. Then $\mathcal{S}$ can be extended to example $(ii)$ in Theorem \ref{thmnoconfigalgemeenk} with $\nu=\langle \nu', \gamma \rangle$, $\pi=\langle \pi', \gamma \rangle$.
\end{enumerate}

If $k-g-1=2$, the set $\mathcal{T}$ is a set of planes in $\PG(n-k+2,q)$ pairwise intersecting in at least a point, i.e. an Erd\H{o}s-Ko-Rado set of planes. In \cite[Section 6]{kneser}, Blokhuis \emph{et al.} classified the maximal Erd{\H{o}}s-Ko-Rado sets $\mathcal{T}$ of planes in $\PG(5,q)$ with $|\mathcal{T}|\geq 3q^4+3q^3+2q^2+q+1$.  In \cite{EKRplanes}, M. De Boeck generalized these results and classified the largest examples of sets of planes pairwise intersecting in at least a point in $\PG(n,q), n\geq 5$. Below we retrace the examples in \cite{kneser}  and \cite{EKRplanes} with size at least $f(k,q)$ and such that there is no point contained in all their elements. For each example, we show that it can be extended to one of the examples discussed in the previous sections, or that it gives rise to a new maximal example.

\begin{enumerate}[a)]
    \item  $\mathcal{T}$ is the set of planes of Example $II$ in \cite{EKRplanes}: Consider a 3-space $\sigma$ and a point $P_0\in \sigma$. Let $\mathcal{T}$ be the set of all planes that either are contained in $\sigma$ or else intersect $\sigma$ in a line through $P_0$. Then $\mathcal{S}$ can be extended to example $(ii)$ in Theorem \ref{thmnoconfigalgemeenk}, with $\zeta$ the $(k+1)$-space spanned by $\sigma$ and $\gamma$, and $\pi_{AB}=\langle \gamma, P_0\rangle$.
    \item $\mathcal{T}$ is the set of planes of Example $III$ in \cite{EKRplanes}: Consider a plane $\pi$, then $\mathcal{T}$ is the set of planes meeting $\pi$ in at least a line. Then $\mathcal{S}$ can be extended to example $(i)$ in Theorem \ref{thmnoconfigalgemeenk} with $\zeta$ the $k$-space spanned by $\pi$ and $\gamma$.
    \item $\mathcal{T}$ is the set of planes of Example $IV$ in \cite{EKRplanes}: Consider a 4-space $\tau$, a plane $\delta \subset \tau$ and a point $P_0\in \delta$. Then $\mathcal{T}$ is the set containing the planes in $\tau$ intersecting $\delta$ in a line, the planes intersecting $\delta$ in a line through $P_0$ and the planes in $\tau$ through $P_0$.  Then we can refer to Subsection \ref{alphasolid} and so $\mathcal{S}$ can be extended to Example \ref{voorbeeldenalgemeen}$(iv)$, with $\rho=\langle \gamma,\tau \rangle$, $\alpha=\langle \gamma,\delta \rangle$ and $\pi=\langle \gamma,P_0 \rangle$.
    \item $\mathcal{T}$ is the set of planes of Example $V$ in \cite{EKRplanes}: Consider a 4-space $\tau$, and a line $l \subset \tau$. Then $\mathcal{T}$ is the set containing the planes through  $l$ and all planes in $\tau$ containing a point of $l$.  Then we can refer to Subsection \ref{alphaplane} and $\mathcal{S}$ can be extended to Example \ref{voorbeeldenalgemeen}$(v)$, with $\rho=\langle \gamma,\tau \rangle$ and $\alpha=\langle \gamma,l \rangle$.
    
    \item $\mathcal{T}$ is the set of planes of Example $VI$ in \cite{EKRplanes}: Let $\tau_1$ and $\tau_2$ be two $4$-spaces such that $\sigma=\tau_1 \cap \tau_2$ is a 3-space. Let $\pi_1$ and $\pi_2$ be two planes in $\sigma$ with intersection line $l_0$ and let $P_1$ and $P_2$ be two different points on $l_0$. Then $\mathcal{T}$ is the set of planes through $l_0$, the planes in $\sigma$, the planes in $\tau_1$ containing a line through $P_1$ in $\pi_1$ or a line through $P_2$ in $\pi_2$, and the planes in $\tau_2$ containing a line
through $P_1$ in $\pi_2$ or a line through $P_2$ in $\pi_1$.  Then by using Section \ref{alpha_4}, Case $1$, $\mathcal{S}$ can be extended to Example \ref{voorbeeldenalgemeen}$(vi)$ with $\rho_i=\langle \gamma,\tau_i \rangle$, $\alpha=\langle \gamma,\sigma \rangle$, $\pi_A=\langle \gamma,\pi_1 \rangle$, $\pi_B=\langle \gamma,\pi_2 \rangle$, $\lambda=\langle \gamma,l_0 \rangle$, $\lambda_A=\langle \gamma,P_1 \rangle$, $\lambda_B=\langle \gamma,P_2 \rangle$ and $P_{AB}$ a point in $\gamma$.

 \item $\mathcal{T}$ is the set of planes of Example $VII$ in \cite{EKRplanes}: Let $\rho$ be a $5$-space. Consider a line $l\subset \rho$ and a $3$-space $\sigma\subset \rho$
disjoint to $l$. Choose three points $P_1$, $P_2$, $P_3$ on $l$ and choose four non-coplanar
points $Q_1$, $Q_2$, $Q_3$, $Q_4$ in $\sigma$. Denote $l_1= Q_1Q_2$, $\bar{l}_1= Q_3Q_4$, $l_2= Q_1Q_3$, $\bar{l}_2= Q_2Q_4$, $l_3= Q_1Q_4$, and $\bar{l}_3= Q_2Q_3$. Then $\mathcal{T}$ is the set containing all planes through $l$ and all planes through $P_i$ in $\langle l,l_i \rangle$ or in $\langle l, \bar{l}_i \rangle$, $i=1,2,3$. Note that this set $\mathcal{S}$ is the set described in Example \ref{voorbeeldenalgemeen}$(ix)$. We can prove the following lemma.
\begin{lemma}\label{maxmaarten7}
The set $\mathcal{S}$ of $k$-spaces described in Example \ref{voorbeeldenalgemeen}$(ix)$ is  a maximal set of $k$-spaces pairwise intersecting in at least a $(k-2)$-space.
\end{lemma} 
\begin{proof}
We have to prove that there exists no $k$-space $E$ in $\PG(n,q)$,  with $\gamma\nsubseteq E$ and so that $E$ meets all elements of $\mathcal{S}$ in at least a $(k-2)$-space. Suppose there exists such a $k$-space $E$. As $\mathcal{S}$ contains all $k$-spaces through the $(k-1)$-space $\langle \gamma,l \rangle$, $E$ contains a $(k-2)$-space $\pi_0$ of $\langle \gamma,l \rangle$, not through $\gamma$. Hence, $\dim(E \cap \gamma)=g-1=k-4$. As $\mathcal{S}$ contains all $k$-spaces through  $\langle \gamma, P_i \rangle$ in the $(k+1)$-space $\langle \gamma,l, l_i \rangle$ (or $\langle \gamma,l, \bar{l}_i \rangle$), $E$ contains a $(k-1)$-space of each of those $(k+1)$-spaces. Consider now the quotient space $\PG(n,q)/ \gamma$, and let $E'=\langle \gamma,E \rangle / \gamma$, $Q_i'=\langle Q_i, \gamma \rangle / \gamma$, $P_i'=\langle P_i, \gamma \rangle / \gamma$, and $l'=\langle l, \gamma \rangle / \gamma$. Then $E'$ is a solid in $\PG(n,q)/ \gamma$ through $l'$ that contains a point of each of the lines $Q'_iQ'_j$, $1\leq i<j\leq 4 $, but this gives a contradiction as $\dim(E')=3$. 
\end{proof}

\item $\mathcal{T}$ is the set of planes of Example $VIII$ in $\PG(n-k+2,q)$ in \cite{EKRplanes}: Consider two solids $\sigma_1$ and $\sigma_2$, intersecting in a line $l$. Take the points $P_1$ and $P_2$ on $l$.  Then $\mathcal{T}$ is the set containing all planes through $l$, all planes through $P_1$ that contain a line in $\sigma_1$ and a line in $\sigma_2$, and all planes through $P_2$ in $\sigma_1$ of $\sigma_2$. Note that this set $\mathcal{S}$ is the set described in Example\ref{voorbeeldenalgemeen}$(viii)$. We  we can prove that the set $\mathcal{S}$ of $k$-spaces is not extendable.

\begin{lemma}\label{maxmaarten8}
The set $\mathcal{S}$ of $k$-spaces described in Example \ref{voorbeeldenalgemeen}$(vii)$ is a maximal set of $k$-spaces pairwise intersecting in at least a $(k-2)$-space.
\end{lemma} 
\begin{proof}
We have to prove that there exist no $k$-space $E$ in $\PG(n,q)$,  with $\gamma\nsubseteq E$ and so that $E$ meets all elements of $\mathcal{S}$ in at least a $(k-2)$-space. Suppose there exists such a $k$-space $E$. As $\mathcal{S}$ contains all $k$-spaces through the $(k-1)$-space $\langle \gamma,l \rangle$, $E$ contains a $(k-2)$-space $\pi_0$ of  $\langle \gamma,l \rangle$, not through $\gamma$. Hence $\dim( \gamma \cap E)=k-4$. As $\mathcal{S}$ contains all $k$-spaces through $\langle \gamma, P_2 \rangle$ in the $(k+1)$-space $\langle \gamma,\sigma_1 \rangle$ (or $\langle \gamma,\sigma_2 \rangle$), $E$ contains a $(k-1)$-space of each of those $(k+1)$-spaces. These two $(k-1)$-spaces, $\alpha_1$ and $\alpha_2$ respectively, span $E$ and meet in a $(k-2)$-space $\pi_0$. Then we show that there exists a $k$-space $A \in \mathcal{S}$, containing $\gamma$, that meets $E$ in precisely a $(k-3)$-space. Consider the quotient space $\PG(n,q)/ \gamma$, and let $E'=\langle \gamma,E \rangle / \gamma$, $\sigma_i'=\langle \sigma_i, \gamma \rangle / \gamma$, $P_i'=\langle P_i, \gamma \rangle / \gamma$, $A'=\langle A, \gamma \rangle / \gamma$ and $l'=\langle l, \gamma \rangle / \gamma=\langle \pi_0, \gamma \rangle / \gamma$. Then $E'$ is a solid in $\PG(n,q)/ \gamma$ through $l'$ that contains planes $\alpha_1'$, $\alpha_2'$ in $\sigma_1'$ and $\sigma_2'$ respectively. Note that $\alpha_1' \cap \alpha_2'=l'$. Let $l_1\in \sigma_1'$ and $l_2 \in \sigma_2'$ be two lines containing $P'_1$ so that $ l_1 \cap \alpha'_1=l_2  \cap \alpha'_2=P_1'$, and let $A'$ be the plane spanned by $l_1$ and $l_2$. Then $E' \cap A'$ is a point in $\PG(n,q)/ \gamma$. Since $\gamma\subseteq A$ and $\gamma\nsubseteq E$ we find that $E \cap A$ is a $(k-3)$-space of $\langle \gamma, P_1 \rangle$ in $\PG(n,q)$, and so these elements of $\mathcal{S}$ meet in a $(k-3)$-space, a contradiction. 
\end{proof}

\item $\mathcal{T}$ is the set of planes of Example $IX$ in $\PG(n-k+2,q)$ in \cite{EKRplanes}: Let $l$ be a line and $\sigma$ a solid skew to $l$. Denote $\langle l, \sigma \rangle$ by  $\rho$. Let $P_1$ and $P_2$ be two points on $l$ and let $\mathcal{R}_1$ and $\mathcal{R}_2$ be a regulus and its opposite regulus in $\sigma$.
Then $\mathcal{T}$ is the set containing all planes through $l$, all planes through $P_1$ in the solid generated by $l$ and a line of $\mathcal{R}_1$, and all planes through $P_2$ in the solid generated by $l$ and a line of $\mathcal{R}_2$. Note that this set $\mathcal{S}$ is the set described in Example \ref{voorbeeldenalgemeen}$(viii)$. We can prove the following lemma. 

\begin{lemma}\label{maxmaarten9}
The set $\mathcal{S}$ of $k$-spaces  described in Example \ref{voorbeeldenalgemeen}$(viii)$  is a maximal set of $k$-spaces pairwise intersecting in at least a $(k-2)$-space.
\end{lemma} 
\begin{proof}
We have to prove that there exists no $k$-space $E$ in $\PG(n,q)$,  with $\gamma\nsubseteq E$, and so that $E$ meets all elements of $\mathcal{S}$ in at least a $(k-2)$-space. Suppose there exists such a $k$-space $E$. Let $\mathcal{R}_1=\{l_1,l_2,\dots, l_{q+1}\}$ and $\mathcal{R}_2=\{\bar{l}_1,\bar{l}_2,\dots, \bar{l}_{q+1}\}$. As $\mathcal{S}$ contains all $k$-spaces through the $(k-1)$-space $\langle \gamma,l \rangle$, $E$ contains a $(k-2)$-space $\pi_0$ of  $\langle \gamma,l \rangle$, not through $\gamma$. Hence, $\dim(\gamma\cap E)=k-4$. As $\mathcal{S}$ contains all $k$-spaces through $\langle \gamma, P_i \rangle$ in the $(k+1)$-spaces $\langle \gamma,l, l_i \rangle$ (or $\langle \gamma,l, \bar{l}_i \rangle$), $E$ contains a $(k-1)$-space of each of those $(k+1)$-spaces.  \textcolor{black}{ Consider now the quotient space $\PG(n,q)/ \gamma$, and let $E'=\langle \gamma,E \rangle / \gamma$, $l_i'=\langle l_i, \gamma \rangle / \gamma$, $\bar{l_i'}=\langle \bar{l}_i, \gamma \rangle / \gamma$, $P_i'=\langle P_i, \gamma \rangle / \gamma$, and $l'=\langle l, \gamma \rangle / \gamma=\langle \pi_0, \gamma \rangle / \gamma$. Then $E'$ is a solid in $\PG(n,q)/ \gamma$ through $l'$ that contains a point of each of the lines $l'_i$ and $\bar{l}'_i$, $1\leq i\leq q+1 $, but this gives a contradiction as $\dim(E')=3$. }
\end{proof}

\end{enumerate}

We see that example $(f),(g)$ and $(h)$ give rise to maximal examples of sets $\mathcal{S}$ of $k$-spaces pairwise intersecting in at least a $(k-2)$-space, described in Example \ref{voorbeeldenalgemeen}$(ix),(vii), (viii)$ respectively. From \cite{EKRplanes}, it follows that the number of elements in $\mathcal{S}$ equals $\theta_{n-k}+6q^2$, $\theta_{n-k}+q^4+2q^3+3q^2$ and $\theta_{n-k}+2q^3+2q^2$ respectively.\\

Finally, if $k-g-1=1$, then $g=k-2$ and so, there is a $(k-2)$-space contained in all solids of $\mathcal{S}$. This case gives rise to Example \ref{voorbeeldenalgemeen}$(i)$.






\section{Main Theorem}

By collecting the results from Propositions \ref{overzicht1}, Theorem \ref{thmnoconfigalgemeenk} and Section \ref{sectionsolidpointk}, we find the following result.

\begin{mtheorem}\label{overzicht2}
Let $\mathcal{S}$ be a maximal set of $k$-spaces pairwise intersecting in at least a $(k-2)$-space in $\PG(n,q)$, $n\geq 2k$, $k\geq3$ . 
Let \begin{align*}
    f(k,q)=\begin{cases}
    3q^4+6q^3+5q^2+q+1 \ \text{    if } k=3, q\geq 2 \text{ or } k=4, q=2 \\
    \theta_{k+1}+q^4+2q^3+3q^2 \ \ \text{    else. }
    \end{cases}
\end{align*}
If $|\mathcal{S}|> f(k,q)$, then $\mathcal{S}$ is one of the  families described in Example \ref{voorbeeldenalgemeen}. 
 Note that for $n>2k+1$, the examples $(i)-(ix)$ are stated in decreasing order of the sizes.
\end{mtheorem}
\begin{proof}
\begin{itemize}
    \item [-] If there is no point contained in all elements of $\mathcal{S}$ and $\mathcal{S}$ contains three $k$-spaces $A,B,C$ with $\dim(A\cap B\cap C)=k-4$, then we distinguished the possibilities for $\mathcal{S}$ depending on the dimension of $\alpha=\langle D\cap \langle A,B \rangle \,|\, D\in \mathcal{S'} \rangle$, where $\mathcal{S'}=\{D \in \mathcal{S} \,|\, D \not \subset \langle A,B \rangle\}$, see Section \ref{confi}. By Proposition \ref{overzicht1}, it follows that $\mathcal{S}$ is one of the examples $ (iv),(v),(vi),(x)$ in Example \ref{voorbeeldenalgemeen}.
    \item [-] If there is no point contained in all elements of $\mathcal{S}$ and if for every three elements $A,B,C$ in $\mathcal{S}$, we have that $\dim(A\cap B\cap C)\geq k-3$, then the only possibilities for $\mathcal{S}$ are described in Example \ref{voorbeeldenalgemeen} $(ii)$ and $(iii)$, see Theorem \ref{thmnoconfigalgemeenk}. 
    \item [-] If there is at least a point contained in all $k$-spaces of $\mathcal{S}$, then we refer to Section \ref{sectionsolidpointk}. Let $\gamma$ be the maximal subspace contained in all $k$-spaces of $\mathcal{S}$, with $\dim(\gamma)=g$. Then $\mathcal{T}=\{D/\gamma \,|\,D \in \mathcal{S}\}$ is a set of  $(k-g-1)$-spaces of $\PG(n-g-1,q) \simeq \PG(n,q)/\gamma$  pairwise intersecting in at least a $(k-g-3)$-space. 
    The only examples of sets $\mathcal{T}$ that give rise to maximal examples of sets of $k$-spaces are described in Section \ref{sectionsolidpointk} in the examples $(f),(g),(h)$ and when $g=k-3$. They correspond to Example \ref{voorbeeldenalgemeen}$(i),(ix),(vii),(viii)$.
\end{itemize}
\end{proof}

\section*{Acknowledgements}

The work of Ago-Erik Riet was partially supported by Estonian Research Council through research grants PSG114 and IUT20-57.\\
The work of Giovanni Longobardi is supported by the Research Project of MIUR (Italian Office for University and Research) "Strutture geometriche, Combinatoria e loro Applicazioni", 2012.\\
Both would like to thank Ghent University for their hospitality, as part of the work was carried out there.\\
The research of Jozefien D’haeseleer is supported by the FWO (Research Foundation Flanders).

\end{document}